\documentclass[11pt]{article}

\usepackage{setspace}
\usepackage[english]{babel}
\usepackage[utf8]{inputenc}
\usepackage{amsmath}
\usepackage{scalerel}
\usepackage{amssymb}
\usepackage{graphicx}
\usepackage{textcomp}
\usepackage{fancyhdr}
\usepackage{amsthm}
\usepackage{amssymb}
\usepackage{graphicx}
\usepackage{appendix}
\usepackage{indentfirst}
\usepackage{mathtools}
\mathtoolsset{showonlyrefs,showmanualtags}
\usepackage{bm}
\usepackage[maxbibnames=99]{biblatex}
\renewbibmacro{in:}{}
\usepackage{csquotes}
\usepackage{enumerate}
\usepackage{xcolor}
\usepackage{subcaption} 
\usepackage{hyperref}
\hypersetup{
    colorlinks=true,
    linkcolor=black,
    citecolor=black,
    pdftitle={Energy growth for systems of coupled oscillators with partial damping}
}

\newcommand{\E}{\bm{\mathrm E}}
\newcommand{\Prob}{\bm{\mathrm{P}}}
\newcommand{\e}{\varepsilon}

\newcommand{\lp}{\left(}
\newcommand{\rp}{\right)}
\newcommand{\la}{\left|}
\newcommand{\ra}{\right|}
\newcommand{\lb}{\left[}
\newcommand{\rb}{\right]}

\def\DS{\displaystyle}

\newtheorem{theorem}{Theorem}[section]
\newtheorem{corollary}[theorem]{Corollary}

\newtheorem{lemma}[theorem]{Lemma}
\newtheorem{proposition}[theorem]{Proposition}

\usepackage[margin=1in]{geometry}
\allowdisplaybreaks
\numberwithin{equation}{section}
\usepackage{graphicx} 

\title{Energy growth for systems of coupled oscillators with partial damping}
\author{Dmitry Dolgopyat\footnote{Dept of Mathematics, University of Maryland, Email: dolgop@umd.edu}, Bassam Fayad\footnote{Dept of Mathematics, University of Maryland, Email: bassam@umd.edu}, Leonid Koralov\footnote{Dept of Mathematics, University of Maryland, Email: koralov@umd.edu}, Shuo Yan\footnote{Dept of Mathematics, Imperial College London, Email: s.yan@imperial.ac.uk}}
\date{}

\begin{document}
\maketitle
\begin{abstract}
{We consider two interacting particles on the circle. The particles are subject to stochastic forcing, which is modeled by white noise. In addition, one of the particles is subject to friction, which models energy dissipation due to the interaction with the environment. We show that, in the diffusive limit, the {absolute value of the velocity of the other} particle converges to the reflected Brownian motion. In other words, the interaction between the particles is asymptotically negligible in the scaling limit. The proof combines averaging for large energies with large deviation estimates for small energies.}
\end{abstract}
\section{Introduction}
Understanding energy transfer in complex systems is a fundamental problem in mathematical physics.
There is vast literature on this subject and, despite rigorous results
in a number of important models, there are still many challenges in our understanding of this phenomenon.
Energy transfer plays a key role in several important phenomena including the following:

(a) {\em Fermi acceleration.} The problem is to describe the motion of a particle in random media where
the particle accelerates due to random energy exchange with the environment. Originally introduced by
Fermi \cite{F49} in order to explain the presence of highly energetic particles in cosmic rays, this model
is a subject of intense research (see, e.g., \cite{LLC80, Do08} and references therein).

(b) {\em Fourier Law.} Here the problem is to describe how the heat emitted from some source(s) spreads 
around a given  domain. In particular, one would like to understand the heat conductivity of different materials.
We refer the readers to \cite{BLRB, Dh08, LLP03, Shp} for a review of this subject.
The Fourier law remains an active area of research, see e.g. \cite{KLO} and  references therein.

(c) {\em Energy cascade in turbulence.} In mathematical terms, the problem is to understand how the energy introduced 
at large scales is transferred to the smaller scales in Hamiltonian PDEs. This problem has certain similarities to
the previous problems as it can be reduced to a system of interacting ODEs after passing to the Fourier basis.
We refer the readers to \cite{CKSTT, EGK16, PT-Quant} for more information on this subject. 

The difficulty of the transfer problems comes from the complex interaction network of multi-particle 
systems (cf. \cite{Lan}). A simpler situation appears in the rarified regime where each particle interacts during long time
intervals only with a fixed small collection of other particles. In an effort to understand such systems, several authors
considered local equilibria for a small number of particles subjected to forcing and dissipation (see, e.g.,
\cite{CEHRB, HM09, HM19} and references therein).
We note that even in the case of systems without forcing, the rate of convergence to equilibrium is not completely understood,
see \cite{DLS24} and references therein for a detailed discussion.

In the present paper, we also deal with a simple system of this form. Namely, we consider two one-degree-of-freedom particles
interacting via a bounded potential. 
In addition, we suppose that each particle gains energy via stochastic forcing that is modeled by white noise. 
The energy dissipation is modeled by friction.  
We suppose that only one of the particles is subject to dissipation. 
The question is whether the energy exchange between the particles is strong enough to  transfer the excess energy to the particle with the friction and eventually remove it from the system so as to ensure that the total energy of the system does not grow on average. 
This turns out not to be the case for the system we consider. 
The mechanism is the following. 
One particle loses energy quickly due to friction, so most of the time it has much smaller energy than the other particle. 
When the other particle has high energy, the relative change in energy is small and so it can be studied using perturbative methods. 
If we neglect the interaction as well as the forcing, then the motion of the second particle is integrable, and so we can average the interaction along the orbit of the unperturbed system. 
The averaged interaction vanishes due to the Hamiltonian nature of the two-particle system. This ensures that the energy exchange is too slow at large energies to keep the total energy of the system finite.

Let us now describe our model more precisely.
Let $V$ be a twice continuously differentiable function on $\mathbb T^2$ and $H$ be defined as
\[
H(r_1,r_2,\theta_1,\theta_2) = \frac{1}{2}(r_1^2+r_2^2)+V(\theta_1,\theta_2), \text{ for }(r_1,r_2,\theta_1,\theta_2)\in \mathbb R^2\times\mathbb T^2.
\]
Consider the two-particle Hamiltonian system in $\mathbb R^2\times\mathbb T^2$, subject to stochastic forcing and energy dissipation:
\begin{equation}
\label{eq:the_system}
    \begin{cases}
        d r_1(t)=-V_1'(\theta_1(t),\theta_2(t))dt+d W_1(t)\\
        d r_2(t)=-V_2'(\theta_1(t),\theta_2(t))dt+dW_2(t)-r_2(t)dt\\
        d\theta_1(t)=r_1(t)dt\\
        d\theta_2(t)=r_2(t)dt
    \end{cases},
\end{equation}where $W_1(t)$ and $W_2(t)$ are two independent Brownian motions, and $V'_1$ and $V'_2$ are the partial derivatives of $V$ with respect to $\theta_1$ and $\theta_2$, respectively. Without the stochastic terms ($dW_1(t)$ and $dW_2(t)$) and the dissipation ($-r_2(t)dt$), the system would preserve the Hamiltonian, and, consequently, $r_1(t)$ and $r_2(t)$ would remain bounded. With the stochastic forcing, the expectation of $|r_1(t)|$ is unbounded, despite the dissipation. Moreover, we have the following convergence result.
\begin{theorem}
\label{thm:main_result}
    For each initial distribution $\mu$ of the processes in \eqref{eq:the_system},
    the process ${|r_1(t\cdot T)|}/{\sqrt{T}}$ converges weakly, as $T\to\infty$, to a Brownian motion starting and reflected at the origin.
\end{theorem}
In a subsequent work, we will  explore multiple coupled oscillators (with the coupling that is not necessarily periodic or bounded), where stochastic forcing and/or energy dissipation may affect some or all of the particles. Different phenomena can be observed in some of the examples, including convergence to a limiting process, as in the current paper, boundedness of the expectation of the energy (if the dissipation prevails), or blow-up of the energy in finite time, all depending on the relative strength of the coupling, forcing, and dissipation. 
We note that recent papers \cite{BL24, BBCL} give sufficient conditions for positive recurrence 
using Lyapunov function techniques. By contrast the case where there is no finite stationary measure is much less studied. 
In the current paper we provide a rigorous treatment of one of the simplest non-trivial cases, so that the present analysis could serve as a blueprint of the method
for future works.
\\

\textbf{Plan of the paper.}
In Section \ref{sec:expansions}, we obtain an accurate expansion of $r_1(t+\sigma)-r_1(t)$, where $\sigma = 1/|r_1(t)|$, assuming $r_1(t)$ is large. Due to the Hamiltonian nature of the  system without noise and friction, the main terms in $r_1(t+\sigma)-r_1(t)$ are given by the variation $W_1(t+\sigma)-W_1(t)$ and an averaged interaction term $-\frac{r_2(t)}{r_1(t)|r_1(t)|}U_2'(\theta_1(t),\theta_2(t))$, where $U$ is a twice continuously differentiable function on $\mathbb T^2$ defined in Lemma~\ref{lem:prelim_expansion}.

In Section~\ref{sec:r2theta2}, we obtain estimates on the supremum and the moments of $r_2(t)$, in Proposition~\ref{prop:sup_of_r_2} and Lemma~\ref{lem:moments}, respectively. 
Proposition~\ref{prop:sup_of_r_2} shows that $r_2(t)$ typically stays relatively small, while
Lemma~\ref{lem:moments} helps us to control the terms with the powers of $r_2$ on the right-hand side of \eqref{eq:expansion} and prove further averaging for the interaction term.

In Section~\ref{sec:r1}, we show that the re-scaled process $r_1(t\cdot T)/\sqrt{T}$ behaves as a Brownian motion until it approaches the origin. This is summarized in Proposition~\ref{prop:outer_martingale_generalized}. 
In Subsection~\ref{SSShortTime}, short time behavior of the process $r_1(t)$ is described in Proposition \ref{prop:expansion-t(R)}. Namely,
$r_1(t)$ behaves as a Brownian motion plus a drift term that is small, provided that $r_1(t)$ is large.
In Subsection~\ref{SSModeratetime}, moderate time behavior of the process is described in Lemma~\ref{lem:choose_beta}, which is proved by using Proposition \ref{prop:expansion-t(R)} and its corollary iteratively.
The proof of Proposition~\ref{prop:outer_martingale_generalized} is given in Subsection~\ref{SSLongtime}.

The times when the re-scaled process $r_1(t\cdot T)/\sqrt{T}$ is near the origin need to be analyzed separately. 
In Section~\ref{sec:exit_time_from_the_origin}, we show that the time spent by the process near the origin is not large, which will allow us to ignore it in the proof of the main result.
The control near the origin is summarized in Proposition~\ref{prop:time_near_zero} and Proposition~\ref{prop:number_excursions}, where we show that the exit time from a neighborhood of zero can be bounded from above, and where we control the number of excursions between two values near the origin.

In Section~\ref{sec:tight}, we prove the tightness of the processes $|r_1(t\cdot T)|/\sqrt{T}$, $T\geq1$, with initial distribution $\mu$ of $(r_1,r_2,\theta_1,\theta_2)$.

Finally, Theorem~\ref{thm:main_result} is proved in Section~\ref{sec:weak_convergence} by putting together Propositions~\ref{prop:sup_of_r_2}, \ref{prop:outer_martingale_generalized}, \ref{prop:time_near_zero}, \ref{prop:number_excursions}, and \ref{prop:tight}, and using standard probability techniques.

\section{Expansions}
\label{sec:expansions}
In this section, we obtain preliminary results on the typical behavior of the processes on short time intervals. In particular, given $r_1(t)$, we give an accurate expansion of $r_1(t+\sigma)$, where $\sigma = 1/|r_1(t)|$.
Let us introduce some notation first.
Let $\mathcal F_t$ be the natural filtration generated by the Brownian motions $W_1(t)$ and $W_2(t)$, and $\tilde{\mathcal F}^T_t$ be the natural filtration generated by the Brownian motions $W_1(t\cdot T)$ and $W_2(t\cdot T)$.
For a random variable $A$ and a function $B$ defined on the parameter space of $A$, we write $A=O(B)$ (or $A\lesssim B$) if there is a constant $M>0$ such that $|A|\leq M |B|$, $A=\Theta(B)$ if there is a constant $M>0$ such that $\|A\|_2\leq M |B|$, and $A=\tilde\Theta(B)$ if there is a constant $M>0$ such that $\|A\|_4\leq M |B|$.
The absolute value on $\mathbb S^1$ is defined to be the shortest distance to $0$.
These are often needed to describe the asymptotic behavior of $A$ (e.g., when the time parameter tends to zero).
We start our analysis by solving explicitly for $r_1(t)$ and $r_2(t)$:
\begin{align}
     r_1(t)&=r_1(0)+W_1(t)-\int_0^t V_1'(\theta_1(s),\theta_2(s))ds=r_1(0)+\tilde\Theta(\sqrt{t})+O(t),\label{eq:solve_r1}\\
     r_2(t)&=e^{-t}r_2(0)-e^{-t}\int_0^t e^{s}V_2'(\theta_1(s),\theta_2(s))ds+e^{-t}\int_0^t e^{s}d{W_2}(s)\label{eq:solve_r2}\\
    &=e^{-t}r_2(0)+O(1-e^{-t})+\tilde\Theta(\sqrt{1-e^{-2t}}).\nonumber
\end{align}
For the other processes, we can write the following expansions:
\begin{equation}
\label{eq2:expansion-theta1}
\begin{aligned}
    \theta_1(t)&=\theta_1(0)+\int_0^t[r_1(0)+\tilde\Theta(\sqrt{s})+O(s)]ds=\theta_1(0)+r_1(0)t+O(t^2)+\tilde\Theta(t^{3/2}),\\
    \theta_2(t)&=\theta_2(0)+\int_0^t\left[e^{-s}r_2(0)+O(1-e^{-s})+\tilde\Theta(\sqrt{1-e^{-2s}})\right]ds\\
    &=\theta_2(0)+r_2(0)(1-e^{-t})+O(t^2)+\tilde\Theta(t^{3/2}).
\end{aligned} 
\end{equation}

As our main result, Theorem~\ref{thm:main_result}, indicates, we aim to show that, away from the origin, the interaction term $V(\theta_1(t),\theta_2(t))$ does not essentially change the behavior of $r_1(t)$ on large time intervals.
Namely, if we let $Z(t)=-\int_0^t V_1'(\theta_1(s),\theta_2(s))ds$, then
\begin{equation}
\label{def:Z}
        r_1(t)=r_1(0)+W_1(t)+Z(t),
\end{equation}
and $Z(t)$ is expected to be small compared with $W_1(t)$ on large time scales.
We start with the behavior of $Z(t)$ after time $\sigma=1/|r_1(0)|$.
Eventually, we'll be interested in asymptotically small values of $\sigma$, but the following result holds for all $\sigma \leq 1$.
\begin{lemma}
\label{lem:prelim_expansion}
If $|r_1(0)|\geq 1$ and $\sigma=1/|r_1(0)|$, then
\begin{align}
    Z(\sigma)&=r_2(0)O(\sigma^2)+O(\sigma^3)+\tilde\Theta(\sigma^{5/2}),\label{eq:intermediate_expansion}\\
    Z(\sigma)&=-\frac{r_2(0)}{r_1(0)|r_1(0)|}U_2'(\theta_1(0),\theta_2(0))+(r_2(0)^2+1)O({\sigma}^3)+\Theta({\sigma}^{5/2}),\label{eq:expansion}
\end{align}where $U(\theta_1,\theta_2)=V(\theta_1,\theta_2)-\int_{\mathbb S^1}V(\theta,\theta_2)d\theta$.
\end{lemma}
It is worth noting that, compared to \eqref{eq:intermediate_expansion}, \eqref{eq:expansion} provides an exact form of the leading term. However, one of the error terms in \eqref{eq:expansion} depends on $r_2(0)$.
\begin{proof}
The proof utilizes the expansions we obtained above and Taylor's expansion of $V_1'(\theta_1(s),\theta_2(s))$ at $(\theta_1(0)+r_1(0)s,\theta_2(0))$ for $s\in[0,\sigma]$.
Below, {to simplify our notation, we denote $\bm r_i=r_i(0)$ and $\bm\theta_i=\theta_i(0)$ {for $i=1,2$}. By \eqref{eq2:expansion-theta1},}
\begin{align}
    Z(\sigma)&=-\int_0^{\sigma}V_1'(\theta_1(s),\theta_2(s))ds\nonumber\\
    &=-\int_0^{\sigma}V_1'\left(\bm \theta_1+\bm r_1s+O(s^2)+\tilde\Theta(s^{3/2}),\bm \theta_2+\bm r_2(1-e^{-s})+O(s^2)+\tilde\Theta(s^{3/2})\right)ds \label{VPrimeInt} \\
    &=-\int_0^{\sigma}V_1'(\bm \theta_1+\bm r_1s,\bm \theta_2)ds+\int_0^{\sigma}(\bm r_2O(1-e^{-s})+O(s^2)+\tilde\Theta(s^{3/2}))ds\nonumber\\
    &=\bm r_2O(\sigma^2)+O(\sigma^3)+\tilde\Theta(\sigma^{5/2}).\nonumber
\end{align}
{This proves \eqref{eq:intermediate_expansion}. Next, \eqref{VPrimeInt} gives}
\begin{align*}
    Z(\sigma)&=-\int_0^{\sigma}V_1'(\bm \theta_1+\bm r_1s,\bm \theta_2)ds\nonumber\\
    &\quad-\int_0^{\sigma}V_{12}''(\bm \theta_1+\bm r_1s,\bm \theta_2)[\bm r_2s+\bm r_2O(s^2)+O(s^2)+\tilde\Theta(s^{3/2})]ds\\
    &\quad+\int_0^\sigma V_{11}''(\bm \theta_1+\bm r_1s,\bm \theta_2)(O(s^2)+\tilde\Theta(s^{3/2}))ds+\bm r_2^2O(\sigma^3)+O(\sigma^5)+\Theta(\sigma^4)\nonumber\\
    &=-\int_0^{\sigma}V_{12}''(\bm \theta_1+\bm r_1s,\bm \theta_2)\bm r_2sds+\bm r_2O(\sigma^3)+O(\sigma^3)+\tilde\Theta(s^{5/2})+\bm r_2^2O(\sigma^3)+\Theta(\sigma^4)\nonumber\\
    &=-\frac{\bm r_2}{\bm r_1}\int_0^\sigma \frac{d}{ds}V_2'(\bm\theta_1+\bm r_1s,\bm\theta_2)\cdot sds+(\bm r_2^2+1)O({\sigma}^3)+\Theta({\sigma}^{5/2})\\
    &=-\frac{\bm r_2}{\bm r_1}\lb V_2'(\bm\theta_1+\bm r_1s,\bm\theta_2)s\bigg|^\sigma_0-\int_0^\sigma V_2'(\bm\theta_1+\bm r_1s,\bm\theta_2)ds\rb+(\bm r_2^2+1)O({\sigma}^3)+\Theta({\sigma}^{5/2})\\
    &=-\frac{\bm r_2}{\bm r_1|\bm r_1|}U_2'(\bm \theta_1,\bm \theta_2)+(\bm r_2^2+1)O({\sigma}^3)+\Theta({\sigma}^{5/2}).\qedhere\end{align*}
\end{proof}

Let us briefly examine the right-hand side of (\ref{eq:expansion}) in order to motivate our further analysis. This formula will be used iteratively to express the increment of $r_1$ on longer time intervals. On short time scales (of order $\sigma = 1/|r_1|$, as above), the difference between the increment of $r_1$ and the Brownian motion is small (of order $1/r^2_1(0)$), provided that $r_2$ is bounded and $r_1$ is large.
However, at larger times, the
cumulative effect of the terms containing $U_2'(\theta_1,\theta_2)$ needs to be controlled. Namely, we'll use an averaging effect to show that the sum of $N$ such terms (each corresponding to a new starting point obtained after a new rotation) is much smaller than $N/r_1^2$, provided that $N$ is large (but not too large so that $r_1$ does not change much on the corresponding time interval).

In the rest of the paper, we make statements with assumptions on the initial condition of $r_1$ and $r_2$. Unless specified otherwise, the results hold uniformly in all the initial conditions of $\theta_1$ and $\theta_2$. In addition, for simplicity, we assume that the function $V$ is bounded by $1$ along with its first and second derivatives. From the proofs, it will be easy to see that the results hold for all twice continuously differentiable $V$.

\section{{Bounds on \texorpdfstring{$r_2$}{r2}}}
\label{sec:r2theta2}
We start this section by giving a result about the supremum of $r_2(t)$ during a large interval of time. It demonstrates that, most of the time, $r_2(t)$ can be expected to be relatively small.
\begin{proposition}
\label{prop:sup_of_r_2}
{Consider a function $D$ such that $D(T)\to\infty$ as $T\to\infty.$ Then for each $t>0$ there is
$T_0(t)$ such that for $T\geq T_0(t)$ we have that for each initial distribution of $r_1(0)$ and $r_2(0)$}
    \begin{equation}
    \label{eq:exponentially_small}
        \Prob(\sup_{0\leq s\leq t\cdot T}|r_2(s)|>|r_2(0)|+D(T))\leq 18tTD(T)e^{-(D(T)-1)^2}
    \end{equation}
\end{proposition}
{Note that the right-hand side of \eqref{eq:exponentially_small}
 converges to $0$ as $T\to\infty$ provided that  $D(T)\geq2\sqrt{\log T}$}.
\begin{proof}
    By \eqref{eq:solve_r2}, we know that
    \begin{equation}
    \label{eq:max_ou_process}
        \Prob(\sup_{0\leq s\leq t\cdot T}|r_2(s)|>|r_2(0)|+D(T))\leq\Prob(\sup_{0\leq s\leq t\cdot T}|r(s)|>D(T)-1),
    \end{equation}
    where $r(s)$ is the Ornstein-Uhlenbeck process defined as:
    \begin{equation}
        dr(s)=-r(s)dt+dW(s),~r(0)=0.
    \end{equation}
    Let $f(x)=\int_0^x e^{y^2}dy$. Then it is not hard to see that $f(r(s))$ is a martingale. In order to apply Doob's martingale inequality, we compute an upper bound for $\E(f(r(s))\vee0)$. Notice that $xf(x)<e^{x^2}$, for all $x\geq0$, and $r(s)\sim\mathcal N(0,\frac{1}{2}(1-e^{-2s}))$. Hence, for all $s$ sufficiently large,
    \begin{align*}
        \E(f(r(s))\vee0)&=\E\chi_{\{0<r(s)<1\}}f(r(s))+\E\chi_{\{r(s)\geq 1\}}f(r(s))\\
        \leq e+\E\chi_{\{r(s)\geq 1\}}e^{r(s)^2}/r(s)
        &\leq e+\int_1^\infty \frac{1}{x}\exp(x^2-\frac{x^2}{1-e^{-2s}})dx\\
        = e+\int_1^\infty \frac{1}{x}\exp(-\frac{x^2}{e^{2s}-1})dx
        &= e+\int_0^\infty\exp(-e^{2y}/(e^{2s}-1))dy\\
        \leq e+\int_0^\infty\exp(-e^{2y-2s})dy
        &\leq e+2s+\int_{2s}^\infty \exp(-e^{2y-2s})dy\\
        \leq e+2s+\int_{2s}^\infty \exp(-y^2)dy
        &\leq 3s.
    \end{align*}
So, by Doob's martingale inequality, for all $T$ sufficiently large,
\begin{align*}
    \Prob(\sup_{0\leq s\leq t\cdot T}r(s)>D(T)-1)&=\Prob(\sup_{0\leq s\leq t\cdot T}f(r(s))>f(D(T)-1))\\
    \leq \E(f(r(tT))\vee0)/f(D(T)-1)
    &\leq 3tT/f(D(T)-1).
\end{align*}
Since $r(s)$ is symmetric, the right-hand side of \eqref{eq:max_ou_process} is at most $6tT/f(D(T)-1)$. The inequality in \eqref{eq:exponentially_small} follows from the fact that $3xf(x)\geq\exp(x^2)$ for all $x$ sufficiently large.
\end{proof}

Now let us focus on the situation where $|r_1(0)|=R$ is large and $|r_2(0)| \leq R^{\gamma}$, with $0 <  {\gamma}<1$, is relatively small. 
This situation is indeed representative, since we expect $|r_1(t\cdot T)|/\sqrt{T}$ to behave like a Brownian motion reflected at the origin, hence it is typically of order $\sqrt{T}$ during time of order $T$, while Proposition~\ref{prop:sup_of_r_2} indicates that $|r_2(t)|$ typically stays below $2\sqrt{\log T}$.
We are interested in the behavior of the processes at time $0\leq t(R)\leq R^{{\alpha}}$, where ${\alpha} > 0$ is a constant less than $2/3$. 
During this time, it is unlikely for $r_1(t)$ to get too far away from its original location. 
Namely, let $c(R)=R^{{\beta}}$ where ${\alpha}/2<{\beta}<1/3$. 
We will see that the probability for $r_1(t)$ to exit from $[R-c(R),R+c(R)]$  within time $t(R)$ is small.
In order to describe the evolution of the process $r_1(t)$, we will use the expansions in Section~\ref{sec:expansions} and divide the time $t(R)$ into small intervals that correspond to the full rotations of $\theta_1(t)$.
Let $\tau_c=\inf\{t:|r_1(t)-r_1(0)|=c(R)\}$.
Define 
\begin{equation}
\label{def:stopping_times}
    \begin{aligned}
    &{\sigma}_0=0,\\
    &{\sigma}_{k+1}={\sigma}_k+\chi_{\{\sigma_k  <  t(R) \wedge \tau_c\}}1/|r_1({\sigma}_k)|,~k\geq0,
\end{aligned}
\end{equation}
and $\tilde n=\inf\{k:\sigma_k\geq t(R)\wedge \tau_c\}\leq t(R)(R+2c(R))$.
The large parameter $R$, the related functions $t(R)$, $c(R)$, the stopping times $\tau_c$, $\sigma_k$, $k\geq0$, and the quantity $\tilde n$ are frequently used in the remainder of the paper. Although we will choose different $t(R)$ and $c(R)$ in different results and their proofs, unless otherwise specified, the stopping times and $\tilde n$ are always defined as above.
Since the error terms in \eqref{eq:expansion} involve $r_2(t)$ and $r_2(t)^2$, we need bounds on $r_2(t)$ at those full rotation time steps $\sigma_k$, $0\leq k\leq\tilde n$, which explains the need for  the next lemma. 
\begin{lemma}
\label{lem:moments}
    Let $|r_1(0)|=R$, $0\leq t(R)\leq R^{{\alpha}}$, where ${\alpha}<2/3$, and $c(R)=R^{{\beta}}$, where ${\alpha}/2<{\beta}<1/3$. Then we have, for all $R$ large and each $k\geq 0$,
    \begin{align}
        \E\chi_{\{k\leq\tilde n\}}|r_2(\sigma_k)|^4&\leq (2|r_2(0)|e^{- k/R}+3)^4,
    \end{align}where $\sigma_k$, $k\geq0$, and $\tilde n$ are defined as in \eqref{def:stopping_times}.
\end{lemma}
\begin{proof}
Since we have good control over $r_2(t)$ at deterministic times, we will compare $r_2(\sigma_k)$ with $r_2( k/R)$.
    Note that $\sigma_{k}$ is close to $ k/R$ if $k\leq\tilde n$. 
    To be more precise,
    \begin{equation}
    \label{eq:sigma_k_2pik/R}
    \begin{aligned}
        &\la\chi_{\{k\leq\tilde n\}}\lp\sigma_k-\frac{k}{R}\rp\ra\leq\chi_{\{k\leq\tilde n\}}\sum_{j=0}^{k-1}\la\sigma_{j+1}-\sigma_j-\frac{1}{R}\ra\\
        &\leq \frac{t(R)}{1/(R+2c(R))}\lp\frac{1}{R-c(R)}-\frac{1}{R}\rp\leq\frac{2t(R)c(R)}{R}.
    \end{aligned}
    \end{equation}
    By the explicit solution in \eqref{eq:solve_r2}, we know that 
    \begin{equation}
    \label{eq:L4normr4}
        \|r_2(t)\|_4\leq |r_2(0)|e^{-t}+2.
    \end{equation}
    Note that for each $k\in\mathbb N$, we have
    \begin{equation}
        r_2(\sigma_k)=r_2\lp\frac{ k}{R}\rp e^{\frac{ k}{R}-\sigma_k}-e^{-\sigma_k}\int_{ k/R}^{\sigma_k}e^sV_2'(\theta_1(s),\theta_2(s))ds+e^{-\sigma_k}\int_{ k/R}^{\sigma_k}e^sdW_s.
    \end{equation}
    (In the case of $\sigma_k< k/R$, $\int_{ k/R}^{\sigma_k}e^sdW_s$ means $-\int_{\sigma_k}^{ k/R}e^sdW_s$.)
    Then it follows that
    \begin{equation}
    \label{R2TimeDif}
        \begin{aligned}
                   \|\chi_{\{k\leq\tilde n\}}(r_2(\sigma_k)-r_2( k/R))\|_4&\leq 
            \|\chi_{\{k\leq\tilde n\}}(r_2\left(k/R\right)-r_2\left(k/R\right)e^{\frac{ k}{R}-\sigma_k})
            \|_4\\
            \quad+\|\chi_{\{k\leq\tilde n\}}e^{-\sigma_k}\int_{ k/R}^{\sigma_k}e^sV_2'(\theta_1(s),\theta_2(s))ds\|_4
            &+\|\chi_{\{k\leq\tilde n\}}e^{-\sigma_k}\int_{ k/R}^{\sigma_k}e^sdW_s\|_4. 
        \end{aligned}
    \end{equation}
    Here the first term on the right-hand side is bounded by 
    \begin{align*}
    \|\chi_{\{k\leq\tilde n\}}(r_2\left(k/R\right)-r_2\left(k/R\right)e^{\frac{ k}{R}-\sigma_k})\|_4&\leq \frac{3t(R)c(R)}{R}\|\chi_{\{k\leq\tilde n\}}r_2( k/R)\|_4\\
    \leq\frac{3t(R)c(R)}{R}\|r_2( k/R)\|_4
    &\leq \frac{3t(R)c(R)}{R}(|r_2(0)|e^{- k/R}+2).
    \end{align*} 
    The second term in \eqref{R2TimeDif} is bounded by $2|\chi_{\{k\leq\tilde n\}}(\sigma_k-k/R)|\leq\frac{4t(R)c(R)}{R}\to0$. The third term is the fourth root of:
    $$
        \E|\chi_{\{k\leq\tilde n\}}e^{-\sigma_k}\int_{k/R}^{\sigma_k}e^sdW_s|^4
        =\E|\chi_{\{k\leq\tilde n\}}e^{k/R-\sigma_k}\int_{k/R}^{\sigma_k}e^{s-k/R}dW_s|^4.
        $$
{Note that the integral in the right-hand side is a time change of a Brownian motion. Therefore}         
  \begin{gather}   
 \E|\chi_{\{k\leq\tilde n\}}e^{-\sigma_k}\int_{k/R}^{\sigma_k}e^sdW_s|^4     
        =\E|\chi_{\{k\leq\tilde n\}}e^{k/R-\sigma_k}\widetilde W(\frac{1}{2}(e^{2(\sigma_k-k/R)}-e^{-2 k/R}))-\widetilde W(\frac{1}{2}(1-e^{-2 k/R}))|^4\\
        \leq \E \sup_{[-\frac{4t(R)c(R)}{R},\frac{4t(R)c(R)}{R}]}|\widetilde W(\frac{1}{2}(e^{2t}-e^{-2 k/R}))
        -\widetilde W(\frac{1}{2}(1-e^{-2 k/R}))|^4
        \to0,
    \end{gather}
    where $\widetilde W$ is another Brownian motion.
    So, we have the estimate on the $L^4$ norm of difference 
    \begin{equation}
    \label{eq:sigmak}
        \left\|\chi_{\{k\leq\tilde n\}}\left(r_2(\sigma_k)-r_2\left(\frac{k}{R}\right)\right)\right\|_4
        =o(|r_2(0)|e^{-k/R}+1),
    \end{equation}
    proving the desired result.
\end{proof}

\section{Behavior of \texorpdfstring{$r_1$}{r1} away from the origin}
\label{sec:r1}
In this section, we will deal with different temporal and spatial scales that are powers of $R$.
We will prove that, if initially $|r_1(0)|$ is large and $|r_2(0)|$ is relatively small, then, after a certain time, the increment of $r_1(t)$ is made up of a Brownian motion and a small correction term.
Recall $Z(t)=-\int_0^t V_1'(\theta_1(s),\theta_2(s))ds$, and
\begin{equation}
\label{eq4:decomposition}
        r_1(t)=r_1(0)+W_1(t)+Z(t).
\end{equation}

{Let $X_t^T=r_1(t\cdot T)/\sqrt{T}$. The main result of this section essentially states that, if we start with
$X_0^T$ of order 1 and stop it when $X$ reaches a fixed (small) level $\e$, then the resulting process converges as $T\to\infty$ to the Brownian motion stopped at the same level (in fact, in order to complete the proof of convergence, we need a tightness result - Proposition \ref{prop:tight} proved in Section \ref{sec:tight}).}
\begin{proposition}
    \label{prop:outer_martingale_generalized}
    For $\e\!>\!0$, let $\eta\!=\!\inf\{s:|X_s^T|=\e\}$. Then, for each $t\!>\!0$ and bounded $f\in\mathbf{C}^3([0,+\infty))$ with bounded derivatives up to order three, we have
    \begin{equation}
        \E\left[f(|X_{\sigma\wedge\eta}^T|)-f(|X_0^T|)-\frac{1}{2}\int_0^{\sigma\wedge\eta}f''(|X_s^T|)ds\right]\to0,
    \end{equation}as $T\to\infty$, uniformly in $|X_0^T|$ on a compact set $K\!\subset\![\e,+\infty)$, $|r_2(0)|\!\leq\!2\log T$, and stopping time $\sigma$ (w.r.t. $\Tilde{\mathcal F}_\cdot^T$) bounded by $t$.
\end{proposition}
{In the proof we will treat the case of positive $r_1$ and deterministic stopping time $\sigma=t$, since the modifications to get the general case are straightforward. Hence, we will focus on proving
    \begin{equation}
    \label{eq:outer_martingale}
        \E\left[f(X_{t\wedge\eta}^T)-f(X_0^T)-\frac{1}{2}\int_0^{t\wedge\eta}f''(X_s^T)ds\right]\to0,
    \end{equation}as $T\to\infty$, uniformly in $X_0^T$ on a compact set $K\subset[\e,+\infty)$ and $|r_2(0)|\leq2\log T$.
{Since the proof is quite long, we divide it into several steps.}}
\subsection{Short time behavior}
\label{SSShortTime}
{The goal of this section is to prove the following key proposition of the paper and its corollary that will be used in the proofs of Proposition~\ref{prop:outer_martingale_generalized} as well as Proposition~\ref{prop:tight} in Section~\ref{sec:tight}.}
\begin{proposition}
\label{prop:expansion-t(R)}
    Suppose that $|r_1(0)|=R$, $|r_2(0)|\leq R^{{\gamma}}$, where ${\gamma}\in(0,\frac{2}{3})$. 
Suppose that ${\alpha}$ and ${\alpha'}$ satisfy ${\gamma}<{\alpha'}<{\alpha}<2/3$. Then, uniformly in  $R^{{\alpha'}}\leq t(R)\leq R^{{\alpha}}$,
    \begin{equation}
        \E Z(t(R))^2=O((t(R)/R)^2),~\E Z(t(R))=o(t(R)/R),\text{ as }R\to\infty.
    \end{equation}
\end{proposition}
\begin{corollary}
\label{cor:moments}
    Suppose that $|r_1(0)|=R$, $|r_2(0)|\leq R^{{\gamma}}$, where ${\gamma}\in(0,\frac{2}{3})$. 
For each ${\alpha}$ and ${\alpha'}$ that satisfy ${\gamma}<{\alpha'}<{\alpha}<2/3$, as $R\to\infty$, uniformly in  $R^{{\alpha'}}\leq t(R)\leq R^{{\alpha}}$,
    \begin{align}
        &\E r_1(t(R))=r_1(0)+o(t(R)/R),\label{eq:increment_of_first_moment}\\
        &\E r_1(t(R))^2=r_1(0)^2+t(R)+o(t(R)),\label{eq:second_moment}\\
        &\E [r_1(t(R))-r_1(0)]^2=t(R)+o(t(R)).\label{eq:increment_of_second_moment}
    \end{align}
\end{corollary}
The proof of Proposition~\ref{prop:expansion-t(R)} is deferred to later after we obtain technical results that will be needed. 
The next lemma provides a rough estimate that shows that $r_1(t)$ cannot exit a large interval too quickly.
\begin{lemma}
\label{lem:exit_probability}
Suppose that $|r_1(0)|=R$, $|r_2(0)|\leq R^{{\gamma}}$, where ${\gamma}\in(0,\frac{2}{3})$. 
For each ${\alpha}$, ${\alpha'}$, and ${\beta}$ that satisfy ${\gamma}<{\alpha'}<{\alpha}<2/3$ and ${\alpha}/2<{\beta}<1/3$, uniformly in  $R^{{\alpha'}}\leq t(R)\leq R^{{\alpha}}$,
\begin{equation}
    \Prob(\tau_c<t(R))=O(({t(R)}/{Rc(R)})^4),\text{ as $R\to\infty$,}
\end{equation}where $c(R)=R^{{\beta}}$ and $\tau_c=\inf\{t:|r_1(t)-r_1(0)|=c(R)\}$.
\end{lemma}
\begin{proof}
    Let the stopping times $\sigma_k$, $k\geq0$, be defined as in \eqref{def:stopping_times}.
    We will use the expansion \eqref{eq:intermediate_expansion} to control the desired probability.
    \begin{align*}
        r_1(\sigma_{\tilde n})&=r_1(0)+W_1(\sigma_{\tilde n})+\sum_{k=0}^{\tilde n-1}r_2(\sigma_k)\cdot O(R^{-2})+O(t(R)R^{-2})+\tilde\Theta(t(R)R^{-3/2})\\
        &=:r_1(0)+W_1(\sigma_{\tilde n})+\mathcal A_1+\mathcal A_2+\mathcal A_3.
    \end{align*}
    By Lemma~\ref{lem:moments} and the fact that $\tilde n<2t(R)R$, we obtain that
        $\|\mathcal A_1\|_4=O(t(R)/R)$.
    Since $\sigma_{\tilde n}<t(R)+2/R$, the probability of $|W_1(\sigma_{\tilde n})|$ being significantly larger than $\sqrt{t(R)}$ is exponentially small. So
    \begin{align*}
        \Prob(|r_1(\sigma_{\tilde n})-r_1(0)|>\frac{4c(R)}{5})
        &\leq \Prob(|W_1(\sigma_{\tilde n})|>c(R)/5)+\sum_{j=1}^3\Prob(|\mathcal A_j|>c(R)/5)\\
        &\lesssim \frac{1}{R^4}+\frac{t(R)^4}{R^4c(R)^4}+0+\frac{t(R)^4}{R^6c(R)^4}\lesssim\frac{t(R)^4}{R^4c(R)^4}.
    \end{align*}
    On the other hand, since on the event $\{\tau_c<t(R)\}$, we have $0\leq\sigma_{\tilde n}-\tau_c<2/R$. Then,
    \begin{align*}
        &\Prob(\tau_c<t(R),|r_1(\sigma_{\tilde n})-r_1(0)|\leq \frac{4c(R)}{5})\\
        &\leq \Prob(\tau_c<t(R),|r_1(\tau_c)-r_1(\sigma_{\tilde n})|> c(R)/5)\\
        &\leq\Prob(\tau_c<t(R),|\int_{\tau_c}^{\sigma_{\tilde n}}V_1'(\theta_1(s),\theta_2(s))ds|+|W_1(\sigma_{\tilde n})-W_1(\tau_c)|>c(R)/5)\\
        &\leq\Prob(\tau_c<t(R),|W_1(\sigma_{\tilde n})-W_1(\tau_c)|>c(R)/6)\\
        &\leq \Prob(\sup_{0\leq s\leq2/R}|W_1(\tau_c+s)-W_1(\tau_c)|>c(R)/6)
    \end{align*}which is exponentially small {w.r.t. $R$}.
    The proof is completed by taking the sum of two probabilities.
\end{proof}

\begin{lemma}
\label{lem:expansion-sigma}
    Suppose that $|r_1(0)|=R$, $|r_2(0)|\leq R^{{\gamma}}$, where ${\gamma}\in(0,\frac{2}{3})$. 
Suppose that ${\alpha}$, ${\alpha'}$, and ${\beta}$ satisfy ${\gamma}<{\alpha'}<{\alpha}<2/3$ and ${\alpha}/2<{\beta}<1/3$. 
Let $\tau_c$, $\sigma_{k}$, $k\geq 0$, and $\tilde n$, be defined as in \eqref{def:stopping_times} with $c(R)=R^{{\beta}}$ and let $\tilde  Z=Z(\sigma_{\tilde n})$.
Then, uniformly in  $R^{{\alpha'}}\leq t(R)\leq R^{{\alpha}}$,
    \begin{equation}
        \E {\tilde  Z}^2=O((t(R)/R)^2),~ \E {\tilde  Z}=o(t(R)/R),\text{ as }R\to\infty.
    \end{equation}
\end{lemma}
\begin{proof}
    The result on the second moment does not require much extra work compared to \eqref{eq:intermediate_expansion}.
    Indeed, from \eqref{eq:the_system} and \eqref{eq:intermediate_expansion}, we know that $\DS{\tilde  Z}=\sum_{k=0}^{\tilde n-1}{\tilde  Z}_k$, where
    \[
    {\tilde  Z}_k = -\int_{\sigma_k}^{\sigma_{k+1}}V_1'(\theta_1(s),\theta_2(s))ds=r_2(\sigma_k)\cdot O(R^{-2})+O(R^{-3})+\tilde\Theta(R^{-5/2}).
    \]
    Thus, by Lemma~\ref{lem:moments}, we obtain
    \[
    \|{\tilde  Z}_k\|_2\lesssim \|r_2(\sigma_k)\|_2/R^2+R^{-5/2}\leq (2|r_2(0)|e^{- k/R}+3)/R^2+R^{-5/2}\lesssim (t(R)e^{- k/R}+1)/R^2.
    \]
    Note that $\tilde n<L(R):=\lfloor t(R)(R+2c(R))\rfloor$. So $\DS\|{\tilde  Z}\|_2\leq\sum_{k=0}^{\tilde n-1}\|{\tilde  Z}_k\|_2=O(t(R)/R)$.

    The result on the first moment requires more delicate treatment.
    Since $\tilde n<L(R)$, by \eqref{eq:expansion},
    \begin{align}
            \E {\tilde  Z}&=\sum_{k=0}^{L(R)}\E\left(\chi_{\{k<\tilde n\}}{\tilde  Z}_k\right)=\sum_{k=0}^{L(R)}\E\left(\chi_{\{k<\tilde n\}}\E({\tilde  Z}_k|\mathcal F_{\sigma_{k}})\right)\nonumber\\
            &=-\sum_{k=0}^{L(R)}\E\left(\chi_{\{k<\tilde n\}}\left[\frac{r_2(\sigma_k)}{r_1(\sigma_k)|r_1(\sigma_k)|}U_2'(\theta_1(\sigma_k),\theta_2(\sigma_k))\right]\right)\label{term:first_moment_one}\\
            &\quad+\sum_{k=0}^{L(R)}\E\left(\chi_{\{k<\tilde n\}}(r_2(\sigma_k)^2+1)\cdot O(1/R^3)\right)+ O(t(R)R^{-3/2}).\label{term:first_moment_two}
        \end{align}
    By Lemma~\ref{lem:moments}, \eqref{term:first_moment_two} is bounded by \[O((r_2(0)^2R+Rt(R))/R^3)+O(t(R)/R^{-3/2})=O(r_2(0)^2/R^2)+o(t(R)/R)=o(t(R)/R)\] since $t(R)\geq R^{{\gamma}}$.
    
    It remains to bound \eqref{term:first_moment_one}.
    The key observation is that, for most values of $0\leq k<\tilde n$, 
    \begin{align*}
        \frac{r_2(\sigma_k)}{|r_1(\sigma_k)|}U_2'(\theta_1(\sigma_k),\theta_2(\sigma_k))&\approx \frac{r_2(k/R)}{|r_1(0)|}U_2'(\theta_1(0),\theta_2(k/R))\\
        &\approx U(\theta_1(0),\theta_2((k+1)/R))-U(\theta_1(0),\theta_2(k/R)),
    \end{align*}and the sum of such terms remains bounded.
    The symbol $\approx$ is used here rather loosely to explain that the terms in \eqref{term:first_moment_one} can be modified without affecting the desired bound.
    This will be explained more precisely later in \eqref{eq4:riemann}.
    We will now replace $r_1(\sigma_k)$ with $r_1(0)$, $\theta_1(\sigma_k)$ with $\theta_1(0)$, $r_2(\sigma_k)$ with $r_2(k/R)$, and $\theta_2(\sigma_k)$ with $\theta_2(k/R)$, for $k<\tilde n$, and estimate the induced errors in the following steps.
    
    \textbf{1.} To start with, we substitute $r_1(\sigma_k)$ by $r_1(0)$ for $k<\tilde n$. Due to the bound on $r_2(\sigma_k)$ in Lemma~\ref{lem:moments}, the difference made to \eqref{term:first_moment_one} will be $o(t(R)/R)$. 
    So, it remains to show that 
    \begin{equation}
    \label{eq:averaging_1}
        \sum_{k=0}^{L(R)}\E\left(\chi_{\{k<\tilde n\}}\left[r_2(\sigma_k)U_2'(\theta_1(\sigma_k),\theta_2(\sigma_k))\right]\right)=o(t(R)R).
    \end{equation}
    
    \textbf{2.} We observe that $\theta_1(\sigma_k)$, $0\leq k< \tilde n<L(R)$, are obtained from each other by almost exact rotations. 
    Namely, for each such $k$, we have that
    \begin{equation}
    \label{eq:substitute-theta1}
        \|\theta_1(\sigma_k)-\theta_1(0)\|_2=O(t(R)/R),
    \end{equation}as $R\to\infty$. Indeed, recalling the equations \eqref{eq:the_system} for $r_1$ and $\theta_1$ and boundedenss of $V'_1$, we obtain:
    \begin{equation}
        \theta_1(\sigma_{k})=\theta_1(0)+\sum_{j=0}^{k-1}\chi_{\{j<\tilde n\}}[r_1(\sigma_j)(\sigma_{j+1}-\sigma_j)+ O(1/R^2)+\int_{\sigma_j}^{\sigma_{j+1}}\int_{\sigma_j}^tdW_1(u)dt].
    \end{equation}
    The total contribution from the first two terms in the bracket is $O(t(R)/R)$ because $\chi_{\{j<\tilde n\}}(\sigma_{j+1}-\sigma_j)=1/|r_1(\sigma_j)|$ and $k<L(R)$. By noting that $\chi_{\{j\geq\tilde n\}}(\sigma_{j+1}-\sigma_j)=0$, \eqref{eq:substitute-theta1} follows from
        \begin{align*}
            \E\lb\sum_{j=0}^{k-1}\chi_{\{j<\tilde n\}}\int_{\sigma_j}^{\sigma_{j+1}}\int_{\sigma_j}^tdW_1(u)dt\rb^2&=\E\lb\sum_{j=0}^{k-1}\int_{\sigma_j}^{\sigma_{j+1}}\int_{\sigma_j}^tdW_1(u)dt\rb^2\\
            =\E\sum_{j=0}^{k-1}\lb\int_{\sigma_j}^{\sigma_{j+1}}\int_{\sigma_j}^tdW_1(u)dt\rb^2
            &\leq\E\sum_{j=0}^{L(R)}\lb\int_{\sigma_j}^{\sigma_{j+1}}\int_u^{\sigma_{j+1}}dtdW_1(u)\rb^2\\
            \leq \sum_{j=0}^{L(R)}\E\lp\int_{\sigma_{j}}^{\sigma_{j+1}}(\sigma_{j+1}-u)^2du\rp
            &\leq\frac{t(R)}{R^2}.
        \end{align*}
    Now we can replace the left-hand side \eqref{eq:averaging_1} by
    \begin{equation}
    \label{eq:averaging_2}
        \sum_{k=0}^{L(R)}\E\chi_{\{k<\tilde n\}}\left[r_2(\sigma_k)U_2'(\theta_1(0),\theta_2(\sigma_k))\right],
    \end{equation}
    with difference estimated by
    \begin{gather}
        \sum_{k=0}^{L(R)}\left|\E\chi_{\{k<\tilde n\}}r_2(\sigma_k)[U_2'(\theta_1(0),\theta_2(\sigma_k))-U_2'(\theta_1(\sigma_k),\theta_2(\sigma_k))]\right|\\
        \leq \sum_{k=0}^{L(R)}\|\chi_{\{k<\tilde n\}}r_2(\sigma_k)\|_2\cdot\|\theta_1(0)-\theta_1(\sigma_k)\|_2\lesssim\sum_{k=0}^{L(R)}(2e^{-k/R}|r_2(0)|+3)\cdot t(R)/R\\
        \lesssim t(R)R\cdot 2t(R)/R=o(t(R)R).
    \end{gather}
    
    \textbf{3.} It remains to estimate \eqref{eq:averaging_2}. Recalling \eqref{eq:sigmak}, \eqref{eq:averaging_2} can be approximated by
    \begin{equation}
        \sum_{k=0}^{L(R)}\E\chi_{\{k<\tilde n\}}\left[r_2(k/R)U_2'(\theta_1(0),\theta_2(\sigma_k))\right]
    \end{equation}
    since $\E|\chi_{\{k\leq\tilde n\}}(r_2(\sigma_k)-r_2(k/{R}))|=o( r_2(0)e^{-k/R}+1)$. 
    Now we discard the first $\sqrt{t(R)}R$ terms {(where $r_2(k/R)$ can be large in $L^2$)}, which will not make a big difference since, by \eqref{eq:solve_r2} (cf. \eqref{eq:L4normr4}),
    \begin{equation}
    \begin{aligned}
        &\left|\sum_{k=0}^{\lfloor\sqrt{t(R)}R\rfloor}\E\chi_{\{k<\tilde n\}}\left[r_2(k/R)U_2'(\theta_1(0),\theta_2(\sigma_k))\right]\right|\\
        &\lesssim\sum_{k=0}^{\lfloor\sqrt{t(R)}R\rfloor}(|r_2(0)|e^{-k/R}+2)=O(r_2(0)R)+O(\sqrt{t(R)}R)=o(t(R)R).
    \end{aligned}
    \end{equation}
    
    \textbf{4.}
    For the tail of the series, we substitute $\theta_2(\sigma_k)$ by $\theta_2(k/R)$, and the difference is
    \begin{align*}
        &\sum_{k=\lceil\sqrt{t(R)}R\rceil}^{L(R)}\E\chi_{\{k<\tilde n\}}\left[r_2(k/R)[U_2'(\theta_1(0),\theta_2(k/R))-U_2'(\theta_1(0),\theta_2(\sigma_k))]\right]\\
        &\lesssim \sum_{k=\lceil\sqrt{t(R)}R\rceil}^{L(R)}\|r_2(k/R)\|_2\cdot\sqrt{\E\chi_{\{k<\tilde n\}}|\theta_2(k/R)-\theta_2(\sigma_k)|}\\
        &\lesssim \sum_{k=\lceil\sqrt{t(R)}R\rceil}^{L(R)}\|r_2(k/R)\|_2\cdot\sqrt{\E\int_{k/R-2t(R)c(R)/R}^{k/R+2t(R)c(R)/R}|r_2(s)|ds}\\
        &\lesssim \sum_{k=\lceil\sqrt{t(R)}R\rceil}^{L(R)} \sqrt{t(R)c(R)/R}=o(t(R)R),
    \end{align*}where we used the Cauchy-Schwarz inequality and the fact that $U_{22}''$ is bounded in the first inequality, $\eqref{eq:sigma_k_2pik/R}$ in the second inequality, and \eqref{eq:solve_r2} (cf. \eqref{eq:L4normr4}) in the third inequality.
    
    Similarly, we can discard the last $4t(R)c(R)$ terms. Now the problem reduces to proving
    \begin{equation}
    \label{eq4:riemann}
        \sum_{k=\lceil\sqrt{t(R)}R\rceil}^{\lfloor t(R)(R-2c(R))\rfloor}\E\chi_{\{k<\tilde n\}}\left[r_2(k/R)U_2'(\theta_1(0),\theta_2(k/R))\right]=o(t(R)R).
    \end{equation}
    We can get rid of the indicator function by the Cauchy-Schwarz inequality, \eqref{eq:L4normr4}, and Lemma~\ref{lem:exit_probability}, since $\Prob(k\geq\tilde n)\leq \Prob(\tau_c<t(R))$ for $k\leq t(R)(R-2c(R))$. { Multiplying both sides by $ 1/R$, we reduce the problem to proving that
    \begin{equation}
        \frac{1}{R}\sum_{k=\lceil\sqrt{t(R)}R\rceil}^{\lfloor t(R)(R-2c(R))\rfloor}\E\left[r_2(k/R)U_2'(\theta_1(0),\theta_2(k/R))\right]=o(t(R)).
    \end{equation}
     The left-hand side is a Riemann sum for the integral
    \begin{gather}
        \E\int_{\lceil\sqrt{t(R)}R\rceil/R}^{\lfloor t(R)(R-2c(R))+1\rfloor/R}r_2(s)U_2'(\theta_1(0),\theta_2(s))ds\\
        =\E U(\theta_1(0),\theta_2(\lfloor t(R)(R-2c(R))+1\rfloor/R))-\E U(\theta_1(0),\theta_2(\lceil\sqrt{t(R)}R\rceil/R)),
    \end{gather}which is bounded.
    The absolute value of the difference between the Riemann sum and the integral is bounded by $\DS \sum_{k=\lceil\sqrt{t(R)}R\rceil}^{\lfloor t(R)(R-2c(R))\rfloor} B_k$, where
        \begin{align}
            B_k&=\E\int_{k/R}^{(k+1)/R}\left|r_2(s)U_2'(\theta_1(0),\theta_2(s))-r_2(k/R)U_2'(\theta_1(0),\theta_2(k/R))\right|ds\\
            &\leq\E\int_{k/R}^{(k+1)/R}|r_2(s)|\cdot\left|U_2'(\theta_1(0),\theta_2(s))-U_2'(\theta_1(0),\theta_2(k/R))\right|ds\\
            &\quad+\E\int_{k/R}^{(k+1)/R}\left|r_2(s)-r_2(k/R)\right|\cdot |U_2'(\theta_1(0),\theta_2(k/R))|ds\\
            &\lesssim \E\int_{k/R}^{(k+1)/R}|r_2(s)(\theta_2(s)-\theta_2(k/R))|ds+\E\int_{k/R}^{(k+1)/R}|r_2(s)-r_2(k/R)|ds\label{eq:align1}\\
            &\leq \E\left[\int_{k/R}^{(k+1)/R}|r_2(s)|ds\right]^2+\E\int_{k/R}^{(k+1)/R}|r_2(s)-r_2(k/R)|ds\label{eq:align2}\\
            &=o\left(\frac{1}{R}\right).\label{eq:align3}
        \end{align}
        Here, \eqref{eq:align1} is due to the boundedness of the first and second derivatives of $U$, and \eqref{eq:align2} is due to \eqref{eq:the_system}. The first term in \eqref{eq:align2} is small due to the Cauchy-Schwarz inequality and the boundedness of the second moment of $r_2(t)$ by \eqref{eq:solve_r2} (cf. \eqref{eq:L4normr4}). And the second term in \eqref{eq:align2} can be shown to be small using  arguments similar to those in Lemma~\ref{lem:moments}.} This completes the proof of the lemma.
\end{proof}
\begin{proof}[Proof of Proposition~\ref{prop:expansion-t(R)}]
    Choose ${\beta}$ such that ${\alpha}/2<{\beta}<1/3$.
    Let $c(R)=R^{{\beta}}$ and stopping times $\tau_c$, $\sigma_k$, $k\geq0$, $\sigma_{\tilde n}$ be defined as in \eqref{def:stopping_times}.
    By \eqref{eq:the_system} and \eqref{def:Z}, we have
    \begin{align*}
    Z(t(R))&=\tilde Z-\int_{t(R)\wedge\tau_c}^{t(R)}V_1'(\theta_1(s),\theta_2(s))ds+\int_{t(R)\wedge\tau_c}^{\sigma_{\tilde n}}V_1'(\theta_1(s),\theta_2(s))ds\\
    &=\tilde Z+\chi_{\{\tau_c<t(R)\}}\cdot O(t(R))+O(1/R).
    \end{align*}
    The desired result follows from Lemma~\ref{lem:exit_probability} and Lemma~\ref{lem:expansion-sigma}.
\end{proof}

\subsection{{Moderate time behavior}}
\label{SSModeratetime}
The results of Subsection~\ref{SSShortTime} concern a relatively small time scale and are enough for the purpose of proving Proposition~\ref{prop:outer_martingale_generalized}.
However, let us first prove the following lemma, which gives estimates on transitions of $r_1(t)$ on a larger time scale than what was considered in Subsection~\ref{SSShortTime}. 
It will be used in Section~\ref{sec:exit_time_from_the_origin} to control the time spent by the process $X_t^T$ near the origin.
\begin{lemma}
\label{lem:choose_beta}
Suppose that ${\gamma}\in(0,\frac{2}{3})$.
    For every ${p}>1$, $\kappa>0$, and for all $R$ sufficiently large (depending on ${p}$ and $\kappa$), define $\tau=\inf\{t:|r_1(t)|^{p}=\frac{1}{2}R^{p}\text{ or }2R^{p}\}\wedge\inf\{t:|r_2(t)|=R^{{\gamma}}\}$. If $|r_1(0)|=R$ and $|r_2(0)|\leq R^{{\gamma}}$, then we have
        \begin{align}
            &\E\tau 
 \leq (2^{1/{p}}-1)(1-(\frac{1}{2})^{1/{p}})R^2 + \kappa R^2,\\
            &\Prob(|r_1(\tau)|^{p}=\frac{1}{2}R^{p})\leq \frac{2}{3}.\label{eq4:probability}
        \end{align}
\end{lemma}
\begin{proof}
    Without loss of generality, we assume that $r_1(0)=R$.
    Fix ${\gamma}<{\alpha}<2/3$ and choose $t(R)=R^{{\alpha}}$.
    Define $\eta_k=k\cdot t(R)$ and $\tilde k=\inf\{k:\eta_k\geq\tau\}$. Note that the difference between $\tau$ and $\eta_{\tilde k}$ is bounded deterministically by $t(R)$. So, by the equation of $r_1(t)$, we have
    \begin{equation}
    \label{eq:tau_eta_first}
        \|r_1(\tau)-r_1(\eta_{\tilde k})\|_1\leq\|r_1(\tau)-r_1(\eta_{\tilde k})\|_2\lesssim t(R). 
    \end{equation}
    By \eqref{eq:second_moment} in Corollary~\ref{cor:moments}, using the Markov property, we have that 
    \begin{equation*}
        \E\chi_{\{k<\tilde k\}}[r_1(\eta_{k+1})^2-r_1(\eta_{k})^2]=\E\chi_{\{k<\tilde k\}} t(R)+o(\E\chi_{\{k<\tilde k\}} t(R)), 
    \end{equation*}
    hence
    \begin{equation}
    \label{eq:sec_moment_eta}
        \E[r_1(\eta_{\tilde k})^2-R^2]=\sum_{k=0}^{\infty}\E\chi_{\{k<\tilde k\}}[r_1(\eta_{k+1})^2-r_1(\eta_{k})^2]=\E\eta_{\tilde k}+o(\E\eta_{\tilde k}). 
    \end{equation}
  Similarly, by \eqref{eq:increment_of_first_moment}, we have
    \begin{equation}
    \label{eq:first_moment_moment}
        \E[r_1(\eta_{\tilde k})-R]=o(\E\eta_{\tilde k}/R). 
    \end{equation}
    We compute the difference in the second moment by the Cauchy-Schwarz inequality and \eqref{eq:tau_eta_first}:
    \begin{equation}
    \label{eq:tau_eta_second}
        \E[r_1(\eta_{\tilde k})^2-r_1(\tau)^2]=\E[2r_1(\tau)+(r_1(\eta_{\tilde k})-r_1(\tau))][(r_1(\eta_{\tilde k})-r_1(\tau)]\lesssim Rt(R)=o(R^2).
    \end{equation}
    Then, we obtain
    \begin{equation}
    \label{eq:second_moment_at_tau}
    \begin{aligned}
        \E[r_1(\tau)-R]^2&=\E r_1(\tau)^2-2R\E r_1(\tau)+R^2\\
       &= \E r_1(\eta_{\tilde k})^2+o(R^2)-2R(R+O(t(R))+o(\E\eta_{\tilde k}/R))+R^2\\
       &= \E\eta_{\tilde k}+o(\E\eta_{\tilde k})+o(R^2)\\
       &=   \E\tau+o(\E\tau)+o(R^2),
    \end{aligned}
    \end{equation}where the second line follows from \eqref{eq:tau_eta_second}, \eqref{eq:tau_eta_first}, and \eqref{eq:first_moment_moment}; the third line follows from \eqref{eq:sec_moment_eta}; and the last line follows from the fact that $|\tau-\eta_{\tilde k}|\leq t(R)$.
    By the definition of $\tau$, we see that 
    \begin{equation*}
        \E[r_1(\tau)-R-(2^{1/{p}}-1)R][r_1(\tau)-R+(1-(\frac{1}{2})^{1/{p}})R]\leq0,
    \end{equation*}since the function inside the expectation in non-positive on $[2^{-1/{p}}R,2^{1/{p}}R]$.
    Opening the bracket, we have
    \begin{equation}
        \E[r_1(\tau)-R]^2-\left(2^{1/{p}}+(\frac{1}{2})^{1/{p}}-2\right)R\cdot\E[r_1(\tau)-R]\leq(2^{1/{p}}-1)(1-(\frac{1}{2})^{1/{p}})R^2.\label{eq4:exit}
    \end{equation}
    As follows from \eqref{eq:tau_eta_first}, \eqref{eq:first_moment_moment}, and the fact that $|\eta_{\tilde k}-\tau|\leq t(R)$, 
    \begin{equation}
    \label{eq:first_moment_at_tau}
        |\E[r_1(\tau)-R]|\leq|\E[r_1(\eta_{\tilde k})-R]|+\E|r_1(\tau)-r_1(\eta_{\tilde k})|=o(\E\tau/R)+O(t(R)),
    \end{equation}
    By combining this with \eqref{eq:second_moment_at_tau} and \eqref{eq4:exit}, we obtain, for ${p}>1$, $\kappa>0$ small enough, and all $R$ sufficiently large,
    \begin{equation*}
        \E\tau\leq(2^{1/{p}}-1)(1-(\frac{1}{2})^{1/{p}})R^2+\kappa R^2<R^2.
    \end{equation*}
    Returning to \eqref{eq:first_moment_at_tau}, we obtain that, for each ${p}>1$, 
    \begin{equation*}
        \E r_1(\tau)
        = R+o(R)\geq\left[\frac{1}{3}\cdot2^{1/{p}}+\frac{2}{3}(\frac{1}{2})^{1/{p}}\right]R,
    \end{equation*}for all $R$ sufficiently large. The desired result \eqref{eq4:probability} follows.
\end{proof}

\subsection{{Long time behavior}}
\label{SSLongtime}
\begin{proof}[Proof of Proposition \ref{prop:outer_martingale_generalized}.]
    {As mentioned earlier, we focus on the proof of \eqref{eq:outer_martingale}.}
    We divide the time into smaller intervals and iteratively apply Corollary~\ref{cor:moments}.
    By Proposition~\ref{prop:sup_of_r_2}, with overwhelming probability, $r_2(t)$ stays relatively small hence satisfies the assumption in Corollary~\ref{cor:moments}.

    Let $\sigma=\inf\{s:r_1(s\cdot T)=2t\cdot T\}\wedge\inf\{s:|r_2(s\cdot T)|=3\log T\}$.
    We introduce this stopping time for technical reasons only, and $\{\sigma\leq t\}$ is not likely to happen.
    Fix $t(x)=x^{1/6}$.
    Define $\tau_0=0$ and inductively define $\tau_{k+1}=\tau_{k}+t(r_1(\tau_k))$, and $\hat n=\inf\{k:\tau_k\geq(t\wedge\eta\wedge\sigma)\cdot T\}$. 
    By the definitions, we immediately see that $\tau_{\hat n}\leq t\cdot T+(2t\cdot T)^{1/6}$.
    Before carrying out our plan of applying Corollary~\ref{cor:moments} repeatedly, we first prove a claim that reduces the proof.
    \newline
    \textbf{Claim.} It is enough to prove \eqref{eq:outer_martingale} with $t\wedge\eta$ replaced by $\tau_{\hat n}/T$.\\
    {\em Proof of the claim.}
    By \eqref{eq:the_system}, the assumption that the derivatives of $V$ are bounded by $1$, and Proposition~\ref{prop:sup_of_r_2}, $\Prob(\sigma<t\wedge\eta)\leq\Prob(\sigma<t)\to0$ as $T\to\infty$. Then, it is enough to show \eqref{eq:outer_martingale} with $t\wedge\eta$ replaced by $t\wedge\eta\wedge\sigma$.
    Now it remains to consider the difference induced during the interval $[t\wedge\eta\wedge\sigma,\tau_{\hat n}/T]$ which has length no more  than $t(r_1(\tau_{\hat n-1}))/T=O(T^{-5/6})$. 
    Then, uniformly in $X_0^T$ on $K$ and $|r_2(0)|\leq2\log T$,
    \begin{align*}
        &\E\left[f(X_{\tau_{\hat n}/T}^T)-f(X_{t\wedge\eta\wedge\sigma}^T)-\frac{1}{2}\int_{{t\wedge\eta\wedge\sigma}}^{\tau_{\hat n}/T}f''(X_s^T)ds\right]\\
        &\lesssim \la\E\left[f(X_{\tau_{\hat n}/T}^T)-f(X_{t\wedge\eta\wedge\sigma}^T)\right]\ra+\E|{\tau_{\hat n}/T}-({{t\wedge\eta\wedge\sigma}})|\\
        &\lesssim\E\left[\frac{1}{\sqrt{T}}|r_1(\tau_{\hat n})-r_1((t\wedge\eta\wedge\sigma)T)|\right]+O(T^{-5/6})= O(T^{-1/3})\to0,
    \end{align*}where the last line follows from \eqref{eq:the_system} and the fact that $|\tau_{\hat n}-(t\wedge\eta\wedge\sigma)T|=O(T^{1/6})$. The claim is proved.
    \qed
    
    Now let us prove \eqref{eq:outer_martingale} with $t\wedge\eta$ replaced by $\tau_{\hat n}/T$.
     We divide {the whole time interval into smaller sub-intervals}  using $\tau_k$, $k\geq0$, and obtain
    \begin{align*}
        &\E\left[f(X_{\tau_{\hat n}/T}^T)-f(X_0^T)-\frac{1}{2}\int_0^{\tau_{\hat n}/T}f''(X_s^T)ds\right]\\
        &=\E\left[f(r_1(\tau_{\hat n})/\sqrt{T})-f(r_1(0)/\sqrt{T})-\frac{1}{2T}\int_0^{\tau_{\hat n}}f''(r_1(s)/\sqrt{T})ds\right]\\
        &=\sum_{k=0}^{\infty}\E\lp\chi_{\{k<\hat n\}}\left[f(r_1({\tau_{k+1}})/\sqrt{T})-f(r_1({\tau_{k}})/\sqrt{T})-\frac{1}{2T}\int_{\tau_{k}}^{\tau_{k+1}}f''(r_1(s)/\sqrt{T})ds\right]\rp\\
        &=\sum_{k=0}^{\infty}\E\lp\chi_{\{k<\hat n\}}\E\left[f(r_1({\tau_{k+1}})/\sqrt{T})-f(r_1({\tau_{k}})/\sqrt{T})-\frac{1}{2T}\int_{\tau_{k}}^{\tau_{k+1}}f''(r_1(s)/\sqrt{T})ds\bigg|\mathcal F_{\tau_k}\right]\rp\\
        &=\sum_{k=0}^{\infty}\E\lp\chi_{\{k<\hat n\}}\E\left[\frac{1}{\sqrt{T}}f'(r_1({\tau_{k}})/\sqrt{T})(r_1({\tau_{k+1}})-r_1({\tau_{k}}))\bigg|\mathcal F_{\tau_k}\right]\rp\\
        &\quad+\sum_{k=0}^{\infty}\E\lp\chi_{\{k<\hat n\}}\E\left[\frac{1}{2T}f''(r_1(\tau_k)/\sqrt{T})(r_1({\tau_{k+1}})-r_1({\tau_{k}}))^2-\frac{1}{2T}f''(\tilde\xi_k)(\tau_{k+1}-\tau_k)\bigg|\mathcal F_{\tau_k}\right]\rp\\
        &\quad+\sum_{k=0}^{\infty}\E\lp\chi_{\{k<\hat n\}}\E\left[\frac{1}{6T\sqrt{T}}f'''(\xi_k)(r_1({\tau_{k+1}})-r_1({\tau_{k}}))^3\bigg|\mathcal F_{\tau_k}\right]\rp\\
        &=:\mathcal T_1+\mathcal T_2+\mathcal T_3,
    \end{align*}where $\xi_k$ and $\tilde\xi_k$ are (random) numbers between $r_1(\tau_k)/\sqrt{T}$ and $r_1(\tau_{k+1})/\sqrt{T}$.
    Recall \eqref{eq:increment_of_first_moment} in Corollary~\ref{cor:moments} and the fact that $\chi_{\{k<\hat n\}}r_1(\tau_k)\geq\e\sqrt{T}$.
    Let us deal with the first term $\mathcal T_1$:
    \begin{align*}
        \mathcal T_1&=\sum_{k=0}^{\infty}\E\lp\chi_{\{k<\hat n\}}\E\left[\frac{1}{\sqrt{T}}f'(\frac{r_1(\tau_k)}{\sqrt{T}})(r_1({\tau_{k+1}})-r_1({\tau_{k}}))\bigg|\mathcal F_{\tau_k}\right]\rp\\
        &\lesssim\sum_{k=0}^{\infty}\E\lp\chi_{\{k<\hat n\}}o(\frac{t(r_1(\tau_k))}{\e T})\rp=o(\E\tau_{\hat n}/T).
    \end{align*}
    Now we recall \eqref{eq:increment_of_second_moment} in Corollary~\ref{cor:moments} and deal with the second term $\mathcal T_2$:
    \begin{align*}
        \mathcal T_2&=\sum_{k=0}^{\infty}\E\lp\chi_{\{k<\hat n\}}\E\left[\frac{1}{2T}f''(\frac{r_1(\tau_k)}{\sqrt{T}})(r_1({\tau_{k+1}})-r_1({\tau_{k}}))^2-\frac{1}{2T}f''(\tilde\xi_k)(\tau_{k+1}-\tau_k)\bigg|\mathcal F_{\tau_k}\right]\rp\\
        &=\frac{1}{2T}\sum_{k=0}^{\infty}\E\lp\chi_{\{k<\hat n\}}f''(\frac{r_1(\tau_k)}{\sqrt{T}})\E\left[(r_1({\tau_{k+1}})-r_1({\tau_{k}}))^2-t(r_1(\tau_k))\bigg|\mathcal F_{\tau_k}\right]\rp\\
        &\quad+\frac{1}{2T}\sum_{k=0}^{\infty}\E\lp\chi_{\{k<\hat n\}}t(r_1(\tau_k))\E\left[f''(\frac{r_1(\tau_k)}{\sqrt{T}})-f''(\tilde\xi_k)\bigg|\mathcal F_{\tau_k}\right]\rp\\
        &\lesssim\frac{1}{2T}\sum_{k=0}^{\infty}\E\lp\chi_{\{k<\hat n\}}o(t(r_1(\tau_k)))\rp+\frac{1}{2T}\sum_{k=0}^{\infty}\E\lp\chi_{\{k<\hat n\}}t(r_1(\tau_k))\E\left[\frac{1}{\sqrt{T}}|r_1(\tau_{k+1})-r_1(\tau_k)|\bigg|\mathcal F_{\tau_k}\right]\rp\\
        &=o(\E\tau_{\hat n}/T)
    \end{align*}since $\tilde\xi_k$ is a number between $r_1(\tau_k)/\sqrt{T}$ and $r_1(\tau_{k+1})/\sqrt{T}$, and \eqref{eq:increment_of_second_moment} is used in the penultimate step. The last step is due to the equation for $r_1$ and the fact that $\chi_{\{k<\hat n\}}t(r_1(\tau_k))=O(\sqrt[6]{T})$.
    Similarly, 
    \begin{align*}
        \mathcal T_3&= \sum_{k=0}^{\infty}\E\lp\chi_{\{k<\hat n\}}\E\left[\frac{1}{6T\sqrt{T}}f'''(\xi_k)(r_1({\tau_{k+1}})-r_1({\tau_{k}}))^3\bigg|\mathcal F_{\tau_k}\right]\rp\\
        &\lesssim\sum_{k=0}^{\infty}\E\lp\chi_{\{k<\hat n\}}t(r_1({\tau_{k}}))^3/(T\sqrt{T})\rp\\
        &\lesssim\sum_{k=0}^{\infty}\E\lp\chi_{\{k<\hat n\}}t(r_1({\tau_{k}}))/T\cdot t(r_1({\tau_{k}}))^2/\sqrt{T})\rp=o(\E\tau_{\hat n}/T),
    \end{align*}since $\chi_{\{k<\hat n\}}t(r_1({\tau_{k}}))=O(\sqrt[6]{T})$ for all $k$.
    Thus, the desired result follows from that $\tau_{\hat n}=O(T)$.
\end{proof}

\section{{Excursions to small values of \texorpdfstring{$r_1$}{r1}}}
\label{sec:exit_time_from_the_origin}
{The goal of this section is to control the process when it becomes small. 
The results that will be needed in the sequel are summarized in Propositions \ref{prop:time_near_zero} and \ref{prop:number_excursions}.}

Fix $\alpha_1=6/7$, $\alpha_2=5/9$, and let ${p}$ in Lemma~\ref{lem:choose_beta} be chosen as $3/2$. 
Note that Lemma~\ref{lem:choose_beta} will be used in this section with ${\gamma}=\alpha_2/\alpha_1<2/3$.
Let $\tau(x)=\inf\{t:|r_1(t)|=x\}$.
To start with, we obtain the upper bound on the time spent by $r_1(t)$ near $0$. More precisely, we consider the time $|r_1(t)|$ spends on the interval $[0,\e\sqrt{T}]$ before exiting, where $\e$ is a small parameter that will be specified later. For time spent below $|\log T|^{\alpha_1}$, we use the boundedness of the drift term to get a crude estimate. Once $|r_1(t)|$ gets large enough, we can apply the result in Lemma~\ref{lem:choose_beta}.
\begin{proposition}
    \label{prop:time_near_zero}
    Let $\zeta=\inf\{t:|r_2(t)|=|\log T|^{\alpha_2}\}$.
    For each $\e>0$ small enough, if $|r_1(0)|\leq\e\sqrt{T}$, $|r_2(0)|\leq|\log T|^{\alpha_2}$, then $\E(\tau(\e\sqrt{T})\wedge\zeta)\lesssim \e^2T$, as $T\to\infty$.
\end{proposition}
Recall that $X_t^T=r_1(t\cdot T)/\sqrt{T}$.
Then, the next task is to prove the following result that will eventually be helpful in showing that the number of excursions from $2\e$ to $\e$ of this process is not too large (see \eqref{eq6:n_excursions}).
\begin{proposition}
\label{prop:number_excursions}
    Assume that $|X_0^T|=2\e$ and $|r_2(0)|<\log T$. Let $\eta=\inf\{t:|X_t^T|=\e\}$. Then there exists $\kappa>0$ such that, for each $\e$ sufficiently small,
    \begin{equation*}
        \E e^{-\eta}\leq1-\kappa\e,
    \end{equation*}
    for all $T$ sufficiently large. 
\end{proposition}
\subsection{Length of one excursion: Proof of Proposition~\ref{prop:time_near_zero}}

\begin{lemma}
    \label{lem:time_close_0}
    As $T\to\infty$, $\E\tau(|\log T|^{\alpha_1})=o(\sqrt[10]{T})$, uniformly in $|r_1(0)|\leq|\log T|^{\alpha_1}$ and $r_2(0)\in\mathbb R$.
\end{lemma}
\begin{proof}
    Since the drift term in the equation for $r_1(t)$ in \eqref{eq:the_system} is bounded by $1$ in absolute value, $\E(\tau(|\log T|^{\alpha_1}))\leq u(0)$, where $u(x)=-\frac{1}{2}e^{2x}+x+\frac{1}{2}e^{2|\log T|^{\alpha_1}}-|\log T|^{\alpha_1}$ is the solution to
    \begin{equation*}
        \begin{cases}
            \frac{1}{2}u''-u'=-1\\
            u'(0)=u(|\log T|^{\alpha_1})=0
        \end{cases}.
    \end{equation*}
    It remains to see that $\DS
        u(0)=-\frac{1}{2}+\frac{1}{2}e^{2|\log T|^{\alpha_1}}-|\log T|^{\alpha_1}=o(\sqrt[10]{T}).$
\end{proof}
\begin{proof}[Proof of Proposition~\ref{prop:time_near_zero}]
For all $k\in\mathbb N$ such that $\lfloor{p}\log_2(|\log T|^{\alpha_1})\rfloor-1\leq k\leq\lfloor{p}\log_2(\e\sqrt{T})\rfloor$, let 
    $$t_k=\E(\text{time spent by $|r_1(t)|$ in }[2^{k/{p}},2^{(k+1)/{p}}]\text{ before }\tau(\e\sqrt T)\wedge\zeta).$$
    \textbf{Claim.} For $\lfloor{p}\log_2(|\log T|^{\alpha_1})\rfloor-1\leq k\leq\lfloor{p}\log_2(\e\sqrt{T})\rfloor$, $t_k\lesssim 2^{2k/{p}-k}(\e\sqrt{T})^{p}$.
    \newline\textit{Proof of the claim.}
    {Be aware that, in the steps below, we use the same notation $\Prob$ and $\E$ in the formulas, however with different initial conditions, specified before the formulas.}
    Note that, since $k\geq \lfloor{p}\log_2(|\log T|^{\alpha_1})\rfloor-1$ and ${\gamma}=\alpha_2/\alpha_1<2/3$, the assumptions in Lemma~\ref{lem:choose_beta} are satisfied if $|r_1(0)|\in[2^{k/{p}},2^{(k+1)/{p}}]$ and $|r_2(0)|\leq|\log T|^{\alpha_2}$.
    Then, using the strong Markov property, one can obtain
    \begin{equation}
    \label{eq5:time}
        \E\tau(2^{(k-1)/{p}})\wedge\tau(2^{(k+2)/{p}})\wedge\zeta\leq 20\cdot 2^{2k/{p}}.
    \end{equation}
    
    Let $K>0$ be specified later, ${\tilde\tau}_1=\inf\{t>K\cdot 2^{2k/{p}}:|r_1(t)|\in[2^{k/{p}},2^{(k+1)/{p}}]\}$, and ${\tilde\tau}_{j+1}=\inf\{t>K\cdot  2^{2k/{p}}+{\tilde\tau}_j:|r_1(t)|\in[2^{k/{p}},2^{(k+1)/{p}}]\}$. 
    Here the idea is that the process $|r_1(t)|$ starting on $[2^{k/{p}},2^{(k+1)/{p}}]$ can make an attempt to reach $2^{(k+2)/{p}}$ within time $K\cdot  2^{2k/{p}}$, and then reaches $\e\sqrt{T}$ before coming back to the interval $[2^{k/{p}},2^{(k+1)/{p}}]$. 
    The stopping times ${\tilde\tau}_j$, $j\geq1$, are used to describe the ``failed'' attempts.
    {Take $K=200$. Then}
    uniformly in $|r_1(0)|\in[2^{k/{p}},2^{(k+1)/{p}}]$,
    \begin{equation}
    \label{eq:goright}
        \begin{aligned}
            &\Prob(\tau(2^{(k+2)/{p}})\wedge\zeta<K\cdot 2^{2k/{p}})\\
            &\geq \Prob(\tau(2^{(k-1)/{p}})\wedge\tau(2^{(k+2)/{p}})\wedge\zeta<K\cdot 2^{2k/{p}},\tau(2^{(k+2)/{p}})\wedge\zeta<\tau(2^{(k-1)/{p}}))\\
            &\geq \Prob(\tau(2^{(k+2)/{p}})\wedge\zeta<\tau(2^{(k-1)/{p}}))-\Prob(\tau(2^{(k-1)/{p}})\wedge\tau(2^{(k+2)/{p}})\wedge\zeta\geq K\cdot 2^{2k/{p}})\\ 
            &\geq \frac{1}{3}\times\frac{1}{3}-\frac{1}{10}
            \geq \frac{1}{100},
        \end{aligned}
    \end{equation} 
    {by  Lemma~\ref{lem:choose_beta} and
\eqref{eq5:time}}. 
    Note that, by \eqref{eq4:probability}, for $|r_1(0)|=2^{(k+2)/{p}}$ and $|r_2(0)|\leq|\log T|^{\alpha_2}$, $\Prob(\tau(\e\sqrt T)\wedge\zeta<\tau(2^{(k+1)/{p}}))$ has a lower bound.
    {Namely, 
    \begin{equation}
    \label{eq4:gambler}
        \Prob(\tau(\e\sqrt T)\wedge\zeta<\tau(2^{(k+1)/{p}}))\geq\frac{2-1}{2^{\lfloor{p}\log_2(\e\sqrt{T})\rfloor-(k+1)}-1}\cdot\frac{1}{3}\geq\frac{2^{k+1}}{(\e\sqrt{T})^{p}}\cdot\frac{1}{3}.
    \end{equation}}
    Then, for each $|r_1(0)|\in[2^{k/{p}},2^{(k+1)/{p}}]$,
        \begin{align*}
            &\Prob({\tilde\tau}_1>\tau(\e\sqrt T)\wedge\zeta)\\
            &\geq\Prob(\tau(2^{(k+2)/{p}})<\zeta,\tau(2^{(k+2)/{p}})<K\cdot 2^{2k/{p}},\tau(\e\sqrt{T})\wedge\zeta<\tilde\tau_1)\\
            &\quad+\Prob(\zeta\leq\tau(2^{(k+2)/{p}}),\zeta<K\cdot 2^{2k/{p}})\\
            &\geq \Prob(\tau(2^{(k+2)/{p}})<\zeta,\tau(2^{(k+2)/{p}})<K\cdot 2^{2k/{p}})\cdot\inf_{\substack{|r_1(0)|=2^{(k+2)/{p}}\\ |r_2(0)|\leq|\log T|^{\alpha_2}}}\Prob(\tau(\e\sqrt T)\wedge\zeta<\tau(2^{(k+1)/{p}}))\\
            &\quad+\Prob(\zeta\leq\tau(2^{(k+2)/{p}}),\zeta<K\cdot 2^{2k/{p}})\cdot 1\\
            &\geq\Prob(\tau(2^{(k+2)/{p}})\wedge\zeta<K\cdot 2^{2k/{p}})\cdot\inf_{\substack{|r_1(0)|=2^{(k+2)/{p}}\\ |r_2(0)|\leq|\log T|^{\alpha_2}}}\Prob(\tau(\e\sqrt T)\wedge\zeta<\tau(2^{(k+1)/{p}}))\\
        &\geq \frac{1}{100}\cdot\frac{2^{k+1}}{(\e\sqrt{T})^{p}}\cdot\frac{1}{3},
        \end{align*}
        where the first inequality holds since the events on the right-hand side are disjoint subsets of the event on the left-hand side; the second inequality holds by the strong Markov property applied at $\tau(2^{(k+2)/{p}})$, and the last inequality follows from the strong Markov property applied at $\tau(2^{(k+2)/{p}})\wedge\zeta$, \eqref{eq:goright}, and \eqref{eq4:gambler}. 
    Thus, for all $k\in\mathbb N$ such that $\lfloor{p}\log_2(|\log T|^{\alpha_1})\rfloor-1\leq k\leq\lfloor{p}\log_2(\e\sqrt{T})\rfloor$,
        \begin{align*}
            t_k&\leq K\cdot 2^{2k/{p}}\sum_{j=1}^\infty\Prob({\tilde\tau}_j<\tau(\e\sqrt T)\wedge\zeta)\\
        &\lesssim 2^{2k/{p}}\sum_{j=1}^\infty\sup_{\substack{|r_1(0)|\in[2^{k/{p}},2^{(k+1)/{p}}]\\|r_2(0)|\leq\log T^{\alpha_2}}}\Prob({\tilde\tau}_1<\tau(\e\sqrt T)\wedge\zeta)^j\lesssim 2^{2k/{p}-k}(\e\sqrt{T})^{p},
        \end{align*}which proves the claim.\qed
        
    Let us continue with the proof of Proposition~\ref{prop:time_near_zero}. Denote $\hat k = \lfloor{p}\log_2(|\log T|^{\alpha_1})\rfloor-1$. 
    Recall that ${p}=3/2<2$. It follows that
    \begin{equation*}
        \sum_{k=\hat k}^{\lfloor{p}\log_2(\e\sqrt{T})\rfloor}t_k\lesssim\e^2 T.
    \end{equation*}
    It remains to consider $\DS\hat t=\E(\text{time spent by $|r_1(t)|$ in }[0,2^{\hat k/{p}}]\text{ before }\tau(\e\sqrt{T})\wedge\zeta).$
    
    Similarly to the way it was done  before, we define following stopping times: 
\[ {\hat\tau}_1=\inf\{t>\sqrt[9]{T}:|r_1(t)|\in[0,2^{\hat k/{p}}]\}, 
\quad\text{and}\quad {\hat\tau}_{j+1}=\inf\{t>\sqrt[9]{T}+{\hat\tau}_j:|r_1(t)|\in[0,2^{\hat k/{p}}]\}.\]
    By Lemma~\ref{lem:time_close_0}, for each $|r_1(0)|\in[0,2^{\hat k/{p}}]$, 
    we have 
    \begin{equation}
    \label{eq4:short-time-exit}
        \Prob(\tau(2^{(\hat k+1)/{p}})<\sqrt[9]{T})\geq\frac{1}{2}.
    \end{equation}
    {As in \eqref{eq4:gambler}, by \eqref{eq4:probability}, for $|r_1(0)|=2^{(\hat k +1)/{p}}$ and $|r_2(0)|\leq|\log T|^{\alpha_2}$, 
    \begin{equation}
    \label{eq4:gambler2}
        \Prob(\tau(\e\sqrt T)\wedge\zeta<\tau(2^{\hat k/{p}}))\geq\frac{2-1}{2^{\lfloor{p}\log_2(\e\sqrt{T})\rfloor-\hat k}-1}\cdot\frac{1}{3}\geq\frac{2^{\hat k}}{(\e\sqrt{T})^{p}}\cdot\frac{1}{3}.
    \end{equation}
    By \eqref{eq4:short-time-exit}, \eqref{eq4:gambler2}, and the strong Markov property applied at $\tau(2^{(\hat k+1)/{p}})\wedge\zeta$, for each $|r_1(0)|\in[0,2^{\hat k/{p}}]$,}
    \begin{align*}
        &\Prob(\hat\tau_1>\tau(\e \sqrt T)\wedge\zeta)\\
        &\geq\Prob\left(\zeta<\sqrt[9]{T},\zeta\leq\tau(2^{(\hat k +1)/{p}})\right)+\Prob\left(\tau(2^{(\hat k +1)/{p}})<\sqrt[9]{T},\tau(2^{(\hat k +1)/{p}})<\zeta,\tau(\e\sqrt{T})\wedge\zeta<\hat\tau_1\right)\\
        &\geq\Prob(\tau(2^{(\hat k+1)/{p}})\wedge\zeta<\sqrt[9]{T})\cdot\inf_{\substack{|r_1(0)|=2^{(\hat k+1)/{p}}\\ |r_2(0)|\leq|\log T|^{\alpha_2}}}\Prob(\tau(\e\sqrt T)\wedge\zeta<\tau(2^{\hat k/{p}}))\\
        &\geq\frac{1}{2}\cdot\frac{2^{\hat k}}{(\e\sqrt{T})^{p}}\cdot\frac{1}{3}.
    \end{align*}
    Thus, by the strong Markov property, $\DS
        \hat t\leq\sqrt[9]{T}\sum_{j=1}^\infty(1-\frac{1}{6}\frac{2^{\hat k}}{(\e\sqrt{T})^{p}})^j\lesssim\sqrt[9]{T}(\e\sqrt{T})^{p}=o(T).$
\end{proof}

\subsection{Number of excursions: Proof of Proposition~\ref{prop:number_excursions}}
We control the number of excursions of $|X_t^T|$ from $2\e$ to $\e$. Proposition~\ref{prop:number_excursions} follows from Lemma~\ref{lem5:number_excursions} by choosing the constant $\kappa=(1-1/\sqrt{e})/5$ .
\begin{lemma}
\label{lem5:number_excursions}
    Assume that $|X_0^T|=2\e$ and $|r_2(0)|<\log T$. Let $\eta=\inf\{t:|X_t^T|=\e\}$. 
  Then, for each $\e$ sufficiently small,
   $\Prob(\eta\geq 1/2)\geq \e/5$ for all $T$ sufficiently large.
\end{lemma}
\begin{proof}[Proof of Lemma \ref{lem5:number_excursions}]
We prove the result in the case where $X_0^T>0$, and the other case can be proved similarly.
{We first prove two claims that will be used with the strong Markov property later.}
\newline
\textbf{Claim 1.} Suppose that $r_1(0)=\sqrt{T}$ and $|r_2(0)|\leq2\log T$.
Let $\sigma=\inf\{s:X_s^T=\e\text{ or }2\}.$ Then, for each $\e$ sufficiently small, $\Prob(\sigma>1/2)>1/4$ for all $T$ sufficiently large.\\
{\em Proof of the claim.} Let $f_\e(x)=-x^2+(2+\e)x-2\e$ on $[\e,2]$. {Note that $f_\e(\e)=f_\e(2)=0$ and $f_\e''=-2$}. Hence, Proposition~\ref{prop:outer_martingale_generalized} gives that
    \begin{equation}
    \label{eq5:convergence1}
        \E\left[f_\e(X_{\sigma\wedge\frac{1}{2}}^T)-f_\e(1)-\frac{1}{2}\int_0^{\sigma\wedge\frac{1}{2}}f_\e''(X_s^T)ds\right]\to0,
    \end{equation}as $T\to\infty$.
    Then, it follows that, for $\e<1/8$ and $T$ sufficiently large so that the left-hand side of \eqref{eq5:convergence1} is less than $1/8$ in absolute value,
    \begin{equation*}
        \E f_\e(X_{\sigma\wedge\frac{1}{2}}^T)\geq f_\e(1)-\E\lp\sigma\wedge\frac{1}{2}\rp-{\frac{1}{8}}\geq 1-\e-\frac{1}{2}-\frac{1}{8}\geq \frac{1}{4}.
    \end{equation*}Since $f_\e(X_\sigma^T)=0$, $\DS
        \Prob(\eta\geq\frac{1}{2})\geq\Prob(\sigma\geq\frac{1}{2})\geq\E f_\e(X_{\sigma\wedge\frac{1}{2}}^T)/\sup_{[\e,2]}f_\e\geq\frac{1}{(2-\e)^2}>\frac{1}{4}.$\qed
\newline
\textbf{Claim 2.} {Suppose that $X_0^T=2\e$ with $0<\e<1$ and $|r_2(0)|<\log T$. Let $\zeta_T=\inf\{s:|r_2(s\cdot T)|=2\log T\}$ and $\sigma'=\inf\{t:X_t^T=1\}\wedge\zeta_T\wedge\frac{1}{2}$. Then, for sufficiently large $T$, $\Prob(\sigma'<\eta)\geq  \e.$}\\
{\em Proof of the claim.} 
    Let $\DS g_\e(x)=\frac{x-\e}{1-\e}$ on $[\e,1]$. By Proposition~\ref{prop:outer_martingale_generalized}, $\DS
        \E\left[g_\e(X_{\sigma'\wedge\eta}^T)-g_\e(X_0^T)\right]\to0$, as $T\to\infty$.
    So, for $T$ sufficiently large, $\DS
        \E g_\e(X_{\sigma'\wedge\eta}^T)\geq g_\e(2\e)-\e^2/(1-\e)=\e.$
    Since $g_\e(\e)=0$ it follows that
   $\DS \Prob(\sigma'<\eta)\geq\E g(X_{\sigma'\wedge\eta}^T)/\sup_{[\e,1]} g_\e\geq\e. $ \qed
        
    We can identify two disjoint subsets of $\{\eta\geq\frac{1}{2}\}$: (i). $\sigma'=1/2<\eta$; (ii). $X_t^T$ reaches $1$ before $\eta\wedge\zeta_T\wedge1/2$ (i.e. $\sigma'=\inf\{t:X_t^T=1\}<\eta\wedge\zeta_T\wedge1/2$), and then spends more time than $1/2$ before reaching $\e$. We obtain by the strong Markov property, for all $T$ sufficiently large,
    \begin{equation*}
    \begin{aligned}
        \Prob\lp\eta\geq\frac{1}{2}\rp&\geq\Prob(\sigma'=\frac{1}{2},\sigma'<\eta)+\Prob(\sigma'<\zeta_T,\sigma'<\frac{1}{2},\eta\geq\frac{1}{2})\\
        &\geq\Prob(\sigma'=\frac{1}{2},\sigma'<\eta)\cdot 1+\Prob(\sigma'<\zeta_T,\sigma'<\frac{1}{2},\sigma'<\eta)\cdot\inf_{r_1(0)=\sqrt{T},|r_2(0)|<2\log T}\Prob(\eta\geq\frac{1}{2})\\
        &\geq\frac{1}{4}(\Prob(\sigma'<\eta)-\Prob(\sigma'=\zeta_T))\geq\e/5,
    \end{aligned}
    \end{equation*}
    by \textbf{Claim 1} and \textbf{Claim 2}, and the fact that $\Prob(\sigma'=\zeta_T)$ is small due to Proposition~\ref{prop:sup_of_r_2}.
\end{proof}

\section{Tightness}
\label{sec:tight}
This section is devoted to the proof of the tightness of the processes $|X_t^T|$, $T\geq1$. 
\begin{proposition}
\label{prop:tight}
    The family of measures on $\mathbf{C}([0,+\infty))$ induced by the processes $|X_t^T|$, $T \geq 1$, with initial distribution $\mu$ of $(r_1,r_2,\theta_1,\theta_2)$, is tight. 
\end{proposition}
\begin{proof}
We verify the conditions of Theorem~7.3 (see also the Corollary after Theorem 7.4) in \cite{billingsley1968convergence}:
\begin{enumerate}[(i)]
    \item For each $\e>0$ and $t>0$, there exist $M>0$ and $T_0>0$ such that
    \begin{equation*}
        \Prob(|X_0^T|>M)<\e,~\text{for all }T>T_0;
    \end{equation*}
    \item For each $\e>0$, $t>0$, and $\kappa>0$, there exist $0<\delta<1$ and $T_0>0$ such that
    \begin{equation*}
        \frac{1}{\delta}\Prob(\sup_{s\leq u\leq s+\delta}\left||X_u^T|-|X_s^T|\right|>\kappa)<\e,~\text{for all }0\leq s\leq t-\delta,~T>T_0.
    \end{equation*}
\end{enumerate}

Part (i) is trivial since we have a fixed initial distribution of $r_1$.
To prove part (ii), we fix $\e>0$, $t>0$, and $\kappa>0$, and notice that for {each} $\delta$ and each $0\leq s\leq t-\delta$, we have
\begin{equation}
\label{Strange}
    \Prob(\sup_{s\leq u\leq s+\delta}\left||X_u^T|-|X_s^T|\right|>\kappa)<\Prob(\tau<s+\delta,\sigma<s+\delta),
\end{equation}where $\tau=\inf\{u\geq s:|X_u^T|\geq2\kappa/3\}$ and $\sigma=\inf\{u\geq \tau:|X_u^T-X_\tau^T|\geq\kappa/3\}$. {(Here we introduce the stopping times to start the process $X^T$ away from the origin in order to apply our previous estimates. The definition of the stopping time $\tau$ also implies the last inequality later in the proof.)}
Moreover, by \eqref{eq:the_system} and Proposition~\ref{prop:sup_of_r_2}, we see that $\Prob(\sup_{0\leq s\leq t\cdot T}|r_1(s)|>T^{2})+\Prob(\sup_{0\leq s\leq t\cdot T}|r_2(s)|>\log T)<1/T$ for all $T$ large enough. 
Thus, by the strong Markov property, it suffices to prove that for any given $2\kappa/3\leq |X_0^\e|\leq T^{3/2}$ (or equivalently $2\kappa\sqrt{T}/3\leq|r_1(0)|\leq T^2$) and $|r_2(0)|\leq\log T$,
\begin{equation}
    \Prob(\tilde\sigma<\delta T)<\e\delta/2,
\end{equation}
where $\tilde\sigma=\inf\{u:|r_1(u)-r_1(0)|\geq\kappa\sqrt{T}/3\}$. Recall the definition of the process $Z(\cdot)$ in \eqref{def:Z}. Then 
\begin{equation}
\label{eq:tightness}
    \Prob(\tilde\sigma<\delta T)<\Prob(\{\sup_{0\leq s\leq\delta T}|W_1(s)|>\kappa\sqrt{T}/6\}\cup\{\sup_{0\leq s\leq\delta T}|Z(s)|>\kappa\sqrt{T}/6\}),
\end{equation}
Let us define $\hat\sigma=\inf\{s:|W_1(s)|\geq\kappa\sqrt{T}/6\}\wedge\inf\{s:|Z(s)|\geq\kappa\sqrt{T}/6\}\wedge\inf\{s:|r_2(s)|\geq\log T\}\wedge\delta T\leq\tilde\sigma$.
Then $\Prob(\tilde\sigma<\delta T)<\Prob(\hat\sigma<\delta T)$.
Since $\Prob(|W_1(\hat\sigma)|\geq\kappa\sqrt{T}/6)$ and $\Prob(|r_2(\hat\sigma)|\geq\log T)$ are exponentially small as $\delta\downarrow0$ and as $T\to\infty$, respectively, it remains to consider $\Prob(|Z(\hat\sigma)|\geq\kappa\sqrt{T}/6)$.
We define $\eta_0=0$, $\eta_{k+1} = \eta_k+ t(r_1(\eta_k))$, where the function $t(x)$ can be chosen as $x^{1/6}$, and $\hat k=\inf\{k:\eta_k\geq\hat\sigma\}$. Then, by Proposition~\ref{prop:expansion-t(R)} and noting that $t(\eta_{\hat k-1})<2\sqrt[3]{T}$, we have for all $T$ sufficiently large,
\begin{equation*}
    \|Z(\hat\sigma)\|_2\leq \|Z(\eta_{\hat k}))\|_2+2\sqrt[3]{T}\leq\frac{2\delta T}{\sqrt[6]{|r_1(0)|-\kappa\sqrt{T}/3}}\cdot \frac{\sqrt[6]{|r_1(0)|+\kappa\sqrt{T}/3}}{|r_1(0)|-\kappa\sqrt{T}/3}+2\sqrt[3]{T}\lesssim\frac{\delta}{\kappa}\sqrt{T}.
\end{equation*}Hence, by Chebyshev's inequality, $\Prob(|Z(\hat\sigma)|\geq\kappa\sqrt{T}/6)\lesssim\delta^2/\kappa^4<\e\delta/4$ if we choose $\delta$ sufficiently small and then $T$ sufficiently large.
\end{proof}
\section{Proof of the main result}
\label{sec:weak_convergence}
{The main idea of the proof is to show that $|X_t^T|$ asymptotically solves the martingale problem for the generator of a Brownian motion reflected at the origin. The times where $X_t^T$ is far from the origin are
handled by Proposition \ref{prop:outer_martingale_generalized},  while the times where $X_t^T$ is small are controlled using the results of Section \ref{sec:exit_time_from_the_origin},
namely, Proposition \ref{prop:time_near_zero} and Proposition \ref{prop:number_excursions}.}
Let us again fix $\alpha_2=5/9$.
{The key step in the proof of Theorem \ref{thm:main_result} is the following estimate.}
\begin{lemma}
\label{lem:martingale_problem}
    For each $t>0$ and bounded $f\in\mathbf{C}^3([0,+\infty))$ with bounded derivatives up to order three, such that $f'_+(0) = 0$, 
    \begin{equation}
        \E [f(|X_{t}^T|)-f(|X_0^T|)-\frac{1}{2}\int_0^{t} f''(|X_s^T|)ds]\to0,~~ \text{as}~ T\to\infty,
    \end{equation}
    uniformly in $X_0^T$ on a compact set and  $|r_2(0)|\leq\frac{1}{2}\log T^{\alpha_2}$.
\end{lemma}
\begin{proof}
    Fix $\delta>0$. Define $\zeta_T=\inf\{s:|r_2(s\cdot T)|=\log T^{\alpha_2}\}$, and a sequence of stopping times: $\eta_0\leq\sigma_1\leq\eta_1\leq\cdots$, by $\eta_0 = 0$, $\sigma_k=\inf\{s\geq\eta_{k-1}:|X_s^T|=2\e\}$, $\eta_k=\inf\{s\geq\sigma_k:|X_s^T|=\e\}$, $k\geq1$, where $\e$ is to be specified later.
    Since $f$ is bounded together with its derivatives and $\Prob(\zeta_T<t)\to0$ as $T\to\infty$ by Proposition~\ref{prop:sup_of_r_2}, it is enough to prove 
    \begin{equation}
    \label{eq:proof_main}
        \E [f(|X_{t\wedge\zeta_T}^T|)-f(|X_0^T|)-\frac{1}{2}\int_0^{t\wedge\zeta_T} f''(|X_s^T|)ds]\to0.
    \end{equation}
    Note that
    \begin{align}
        &\E [f(|X_{t\wedge\zeta_T}^T|)-f(|X_0^T|)-\frac{1}{2}\int_0^{t\wedge\zeta_T} f''(|X_s^T|)ds]\nonumber\\
        &=\E\left[f(|X_{\sigma_{1}\wedge t\wedge{\zeta_T}}^T|)-f(|X_0^T|)-\frac{1}{2}\int_{0}^{\sigma_1\wedge t\wedge{\zeta_T}} f''(|X_s^T|)ds\right]\label{term:main_three}\\
        &\quad+\E\sum_{k=1}^{\infty}\chi_{\{\eta_k<t\wedge{\zeta_T}\}}\left[f(|X_{\sigma_{k+1}\wedge t\wedge{\zeta_T}}^T|)-f(|X_{\eta_{k}}^T|)-\frac{1}{2}\int_{\eta_{k}}^{\sigma_{k+1}\wedge t\wedge{\zeta_T}} f''(|X_s^T|)ds\right]\label{term:main_one}\\
        &\quad+\E\sum_{k=1}^{\infty}\chi_{\{\sigma_k<t\wedge{\zeta_T}\}}\left[f(|X_{\eta_{k}\wedge t\wedge{\zeta_T}}^T|)-f(|X_{\sigma_{k}}^T|)-\frac{1}{2}\int_{\sigma_{k}}^{\eta_{k}\wedge t\wedge{\zeta_T}} f''(|X_s^T|)ds\right].\label{term:main_two}
    \end{align}
    
    If $\e$ is small enough and $|X_0^T|<2\e$, then $|f(|X_{\sigma_{1}\wedge t\wedge{\zeta_T}}^T|)-f(|X_0^T|)|$ is small enough due to the boundedness of $f'$, and the integral in \eqref{term:main_three} is also small due to the boundedness of $f''$ and Proposition~\ref{prop:time_near_zero} when $T$ is sufficiently large. 
    On the other hand, if $|X_0^T|\geq2\e$ for a given $\e$, then, by Proposition~\ref{prop:outer_martingale_generalized}, we can choose $T$ large enough such that \eqref{term:main_three} is small. 
    Thus, \eqref{term:main_three} can be bounded by $\delta/3$ for each small $\e$ and all large $T$ depending on $\e$.
    
    By Proposition~\ref{prop:number_excursions}, we have that for each $k\geq 0$, by the strong Markov property,
    \begin{equation}
    \label{eq6:n_excursions}
    \begin{aligned}
        \Prob(\eta_k<t\wedge{\zeta_T})\leq e^t\E(\chi_{\{\eta_k<{\zeta_T}\}}e^{-\eta_k})&\leq e^t(1-\kappa\e)^k,\\
        \Prob(\sigma_k<t\wedge{\zeta_T})\leq \Prob(\eta_{k-1}<t\wedge{\zeta_T})&\leq e^t(1-\kappa\e)^{k-1}.
    \end{aligned}
    \end{equation}
    It follows that
    \begin{align}
        \sum_{k=0}^\infty\Prob(\eta_k<t\wedge{\zeta_T})&\leq\frac{e^t}{\kappa\e},\label{term:num_excursions1}\\
        \sum_{k=1}^\infty\Prob(\sigma_k<t\wedge{\zeta_T})&\leq\frac{e^t}{\kappa\e}.\label{term:num_excursions}
    \end{align}
    Then, we can bound the absolute value of \eqref{term:main_one}, using the strong Markov property, by
    \begin{align}
        \sum_{k=1}^{\infty}\Prob(\eta_k<t\wedge{\zeta_T})\cdot \mathcal S,\label{term6:first}
    \end{align}
    where\[\mathcal S=\sup_{\substack{|r_1(0)|=\e\sqrt{T}\\|r_2(0)|\leq\log T^{\alpha_2}}}\E\left|[f(|X_{\sigma_{1}\wedge t\wedge{\zeta_T}}^T|)-f(|X_0^T|)-\frac{1}{2}\int_{0}^{\sigma_{1}\wedge t\wedge{\zeta_T}} f''(|X_s^T|)ds]\right|.\]
    We can choose $\e$ small enough and then $T$ large enough such that $\mathcal S/\e$ is small since $f'_+(0)=0$, $||X_{\sigma_{1}\wedge t\wedge{\zeta_T}}^T|-\e|\leq\e$, and $\E\sigma_1\wedge t\wedge\zeta_T=O(\e^2)$, as $T\to\infty$, by Proposition~\ref{prop:time_near_zero}.
    Therefore, \eqref{term6:first} can be made smaller than $\delta/3$ for small $\e$ and all large $T$, by \eqref{term:num_excursions1}.
    
    {We can estimate \eqref{term:main_two} similarly. Namely, for each $\e>0$,}
    \begin{align*}
        &\left|\E\sum_{k=1}^{\infty}\chi_{\{\sigma_k<t\wedge{\zeta_T}\}}[f(|X_{\eta_{k}\wedge t\wedge{\zeta_T}}^T|)-f(|X_{\sigma_{k}}^T|)-\frac{1}{2}\int_{\sigma_{k}}^{\eta_{k}\wedge t\wedge{\zeta_T}} f''(|X_s^T|)ds]\right|\\
        &\leq \sum_{k=1}^{\infty}\left|\E\lp\chi_{\{\sigma_k<t\wedge{\zeta_T}\}}[f(|X_{\eta_{k}\wedge t\wedge{\zeta_T}}^T|)-f(|X_{\sigma_{k}}^T|)-\frac{1}{2}\int_{\sigma_{k}}^{\eta_{k}\wedge t\wedge\zeta_T} f''(|X_s^T|)ds]\rp\right|\leq \delta/3,
    \end{align*}for all $T$ sufficiently large, by the strong Markov property, \eqref{term:num_excursions}, and Proposition~\ref{prop:outer_martingale_generalized}.
    The result follows by taking $\e$ sufficiently small, and then $T$ sufficiently large depending on $\e$.
\end{proof}

\begin{proof}[Proof of Theorem~\ref{thm:main_result}]
    By the tightness of the family of processes $|X_t^T|$ established in Proposition~\ref{prop:tight}, for every sequence $\hat T_n\to\infty$, there is a subsequence $T_n\to\infty$ such that $X_t^{T_{n}}$ converges weakly to a continuous process $Y_t$. 
    The desired weak convergence follows if we prove that, no matter which sequence $\{\hat T_n\}$ and $\{T_n\}$ we choose, $Y_t$ always coincides in distribution with the Brownian motion starting and reflected at the origin. 
    Let $S=[0,+\infty)$, $C_0(S)$ be the set of continuous functions that converge to $0$ at $+\infty$, $A=\frac{1}{2}\Delta$ with the domain $\mathcal D(A)=\{f\in C_0(S):f'_+(0)=0, f''\in C_0(S)\}$, and $\mathcal D$ is the set of all functions in $\mathcal D(A)$ that have bounded derivatives up to order three. 
    We will show that $Y_t$ is a solution to the martingale problem for $(A|_{\mathcal D},0)$, i.e., for each $f\in\mathcal D$ and $0\leq t_1\leq t_2$, with $\mathcal F_{t}^{Y_\cdot}$ being the filtration generated by $Y_t$,
    \begin{equation}
    \label{eq:mg_problem}
        \E[f(Y_{t_2})-f(Y_{t_1})-\frac{1}{2}\int_{t_1}^{t_2}f''(Y_t)dt\big|\mathcal F_{t_1}^{Y_\cdot}]=0,\quad Y_0=0.
    \end{equation}
    
    {It is easy to see that 
  the pair $(C_0(S), \mathcal D)$ satisfies the following conditions:}
    \begin{enumerate}[(i)]
        \item $\mathcal D$ is dense in $C_0(S)$;
        \item There exists $\lambda>0$ such that $\mathrm{Range}(\lambda-A|_{\mathcal D})$ is dense in $C_0(S)$;
        \item For each pair of measures $\nu_1,\nu_2$ on $S$, the equality $\int_S fd\nu_1=\int_S fd\nu_2$ for all $f\in C_0(S)$ implies $\nu_1=\nu_2$.
    \end{enumerate}
    {Therefore, Theorem 4.1 in Chapter 4 of \cite{ethier_kurtz} implies that the solution to the martingale problem \eqref{eq:mg_problem} is unique. }
    
    {It remains to prove every limiting process satisfies \eqref{eq:mg_problem}.} It is easy to see that $|X_0^T|\to0$ as $T\to\infty$. Therefore, it is sufficient to prove that, for each $f\in\mathcal D$, $0\leq s_1<...<s_k\leq t_1$, and $g_1,...,g_k\in C_0(S)$,
    \begin{equation*}
        \E\left[\prod_{i=1}^k g_i(Y_{s_i})(f(Y_{t_2})-f(Y_{t_1})-\frac{1}{2}\int_{t_1}^{t_2}f''(Y_t)dt)\right]=0.
    \end{equation*}
    Since $X_t^{T_n}$ converges to $Y_t$ weakly, it is enough to prove
    \begin{equation*}
        \E\left[\prod_{i=1}^k g_i(|X_{s_i}^{T_n}|)(f(|X_{t_2}^{T_n}|)-f(|X_{t_1}^{T_n}|)-\frac{1}{2}\int_{t_1}^{t_2}f''(|X_{t}^{T_n}|)dt)\right]\to0,
    \end{equation*}
    which follows from Proposition~\ref{prop:sup_of_r_2}, Proposition~\ref{prop:tight}, Lemma~\ref{lem:martingale_problem}, and the strong Markov property.
\end{proof}

\section*{Acknowledgements}
We are grateful to Jacob Bedrossian for introducing us to this problem.
Dmitry Dolgopyat was supported by the NSF grants DMS-1956049 and DMS-2246983.
Bassam Fayad was supported by the NSF grant DMS-2101464.
Leonid Koralov was supported by the NSF grant DMS-2307377 and by the Simons Foundation Grant MP-TSM-00002743.
Shuo Yan was supported by the NSF grant DMS-1956049.
\printbibliography
\end{document}